\documentclass[]{article}
\usepackage[utf8]{inputenc}
\usepackage{amsmath}
\usepackage{amssymb}
\usepackage{amsthm}
\usepackage{amsopn}
\usepackage{xcolor}
\usepackage{nicefrac}
\usepackage{enumitem}
\usepackage{color}

\newtheorem{theorem}{Theorem}
\newtheorem{lemma}{Lemma}
\newtheorem{remark}{Remark}
\newtheorem{proposition}[theorem]{Proposition}
\newtheorem{corollary}[theorem]{Corollary}

\usepackage[hidelinks]{hyperref}

\newcommand{\R}{\mathbb{R}}
\newcommand{\RN}{\mathbb{R}^N}
\newcommand{\locLip}{W^{1,\infty}_{\textit{loc}}}

\title{Large-time behavior of unbounded solutions of viscous Hamilton-Jacobi Equations in $\mathbb{R}^N$.}
\author{Guy Barles, Alexander Quaas and Andrei Rodr\'iguez\footnote{Corresponding author.}}

\begin{document}

\maketitle

\begin{abstract} We study the large-time behavior of bounded from below solutions of
parabolic viscous Hamilton-Jacobi Equations in the whole space $\mathbb{R}^N$ in the case of superquadratic Hamiltonians.
Existence and uniqueness of such solutions are shown in a very general framework, namely when the source term and the initial
data are only bounded from below with an arbitrary growth at infinity. Our main result is that these solutions have an ergodic behavior
when $t\to +\infty$, i.e., they behave like $\lambda^*t + \phi(x)$ where $\lambda^*$ is the maximal ergodic constant and $\phi$ is a
solution of the associated ergodic problem. The main originality of this result comes from the generality of the data: in particular, the initial data
may have a completely different growth at infinity from those of the solution of the ergodic problem.
\end{abstract}

\vspace*{1em}



\section{Introduction and main results}

In this article, we consider the parabolic viscous Hamilton-Jacobi Equation
\begin{align}
u_t - \Delta u + |Du|^m = f(x) &{}\quad\textrm{in } \RN\times (0,+\infty), \label{vhj_whole}\\
u(x, 0) = u_0(x) &{}\quad\textrm{in } \RN,\label{initialData_whole}
\end{align}
in the case when $m>2$ and $f, u_0 \in C(\RN)$ are bounded from below. Further assumptions on $f$ are 
stated as required. We immediately point out that, by changing $u(x,t)$ in $u(x,t)+C_1t+C_2$ for $C_1,C_2>0$, we change $f$ in $f+C_1$ and $u_0$ in $u_0+C_2$ and therefore we may assume without loss of generality that $f$ is larger than any constant and $u_0$ is nonnegative, two properties that we 
use later on.

Our main interest is the large-time behavior of the solutions of \eqref{vhj_whole}-\eqref{initialData_whole}
but the first question concerns the existence and uniqueness of bounded from below solutions. Of course, 
the difficulty comes from the very general framework we wish to handle, namely the case when $f$ and $u_0$
are bounded from below with an arbitrary growth at infinity. We do not want to enter into details in this introduction
but we just point out that we use in a key way approximations by problems set on bounded domains, and in particular
on the state-constraint problem (cf. \cite{tabet2010large}) together with various regularity results and a comparison result,
i.e. a Maximum Principle type result, in this general class of solutions.

Coming back to the large-time behavior, there is a vast literature for nonlinear parabolic equations, to the extent that it is
practically impossible to list all relevant works. We point out anyway that, for such viscous Hamilton-Jacobi equations
set in the whole space $\RN$, the non-compactness of the domain and the generality of the data are well-known difficulties;
on the contrary, in the periodic case, the methods introduced in \cite{BS} provide rather general answers for a large class of
equations. On the other hand, problems set in bounded
domains with Dirichlet boundary conditions create different type of difficulties, mainly connected to the fact that the ``natural'' associated
stationary problems may have no solution. This is the reason why the ergodic constant and the ergodic problem come out. We refer to \cite{tabet2010large} which completely solves this problem in the superquadratic case (see also \cite{barles2010large} for the subquadratic one).

We briefly recall the results of \cite{tabet2010large} in the superquadratic case since it shares some similarities with our problem.
There are two possibilities for the large time behavior of solutions
of the Dirichlet problem in $\Omega$ with a Dirichlet boundary condition $g$: if the corresponding stationary equation
\begin{equation}\label{steadysourceplus}
\left\{\begin{array}{ll}
-\Delta v + |Dv|^m = f(x) &\quad\textrm{ in } \Omega\\
v = g &\quad\textrm{ on } \partial\Omega
\end{array}\right.
\end{equation}
has a bounded subsolution, then there exists a solution $u_\infty$ of \eqref{steadysourceplus} and $u(x,t)\rightarrow u_\infty$ on $\overline{\Omega}$. If \eqref{steadysourceplus} fails to have bounded subsolutions, one must introduce the \emph{ergodic problem} with state-constraint boundary conditions
\begin{equation}\label{ergodic_bounded}
\left\{\begin{array}{ll}
-\Delta v + |Dv|^m = f(x) + c &\quad\textrm{ in } \Omega,\\
-\Delta v + |Dv|^m \geq  f(x) + c &\quad\textrm{ in } \partial\Omega.
\end{array}\right.
\end{equation}
Here $c\in\mathbb{R}$ is called \emph{ergodic constant}, and is an unknown in problem \eqref{ergodic_bounded}, as is $v$. Existence and uniqueness of solutions $(c,v)$ of \eqref{ergodic_bounded} are studied in \cite{lasry1989nonlinear}: $c\in\mathbb{R}$ is unique while $v\in C(\overline{\Omega})$ is unique up to an additive constant. Convergence of $u(x,t) + ct$ to $v$ where $(c,v)$ is a solution of \eqref{ergodic_bounded}, is then obtained.

The behavior of solutions to \eqref{vhj_whole} in the subquadratic case, $m \leq 2$, studied in \cite{barles2010large}, is more complicated, as it depends now on whether $1<m\leq\frac{3}{2}$ or $\frac{3}{2}<m\leq2$. It also becomes necessary to introduce the following problem, also studied in \cite{lasry1989nonlinear}, as an analogue of \eqref{steadysourceplus} and \eqref{ergodic_bounded}
\begin{equation}\label{stateconstraints}
\left\{\begin{array}{ll}
-\Delta v + |Dv|^p = f(x) + c &\quad\textrm{ in } \Omega,\\
v(x)\rightarrow+\infty &\quad\textrm{ as } x\rightarrow\partial\Omega.
\end{array}\right.
\end{equation} 
Both \cite{barles2010large} and \cite{tabet2010large} contain a complete study of \eqref{ergodic_bounded} and \eqref{stateconstraints}
in the context of viscosity solutions.

The main difficulty in our study with respect to the previously cited works is the fact that the domain is unbounded, and that under our
assumptions for $f$ and $u_0$, the solutions of \eqref{vhj_whole}-\eqref{initialData_whole} may have any growth when $|x| \to \infty$.

The work \cite{ichihara2012large} addresses the problem of large-time behavior of unbounded solutions of 
\eqref{vhj_whole}-\eqref{initialData_whole} mainly with probabilistic techniques with $x$-dependent Hamiltonians. It is based on explicit
representation formulas for the solution of \eqref{vhj_whole}-\eqref{initialData_whole} which comes from the underlying stochastic
optimal control problem (see,  e.g. \cite{bardi2008optimal}, \cite{fleming2006controlled} for standard references on this topic). To the best of our 
knowledge, \cite{barles2016unbounded} and the present work are the first to address the problem exclusively within the framework of viscosity solutions 
and  PDE techniques. We also point out that our result holds in greater generality with respect to the data. Further comments on this are provided after the statement of our main result.

A number of works have addressed the problem set in the whole space (see e.g. \cite{barles2001space}, \cite{benachour2004asymptotic}, \cite{biler2004asymptotic}, \cite{gallay2007asymptotic}, \cite{iagar2010asymptotic}, \cite{laurenccot2009non}, and the works cited therein), while assuming some restriction on the behavior of the initial data at infinity, which include integrability or decay conditions, periodicity or compact support.

Concerning unbounded solutions on the whole space for \emph{first-order} equations, we refer the reader to the works \cite{barles2006ergodic}, \cite{ichihara2009long}, \cite{ishii2008asymptotic}, and especially to the review \cite{ishii2009asymptotic}, to name a few. We also mention a recent result for first-order Hamilton-Jacobi equations, \cite{barles2017large}, in a situation which is similar to ours, even if our results are obtained by an completely
different approach: under the standard assumptions on the Hamiltonian, such as coercivity and locally Lipschitz continuity, the expected convergence result holds if the initial data is bounded from below. It is also shown in \cite{barles2017large} that in certain special cases, the expected large-time behavior occurs for initial data no longer bounded from below, but fails in general. See \cite{barles2017large} for details.

Finally, we mention the recent work \cite{arapostathis2019uniqueness} which---although dealing with a stationary version of \eqref{vhj_whole}, and in the subquadratic case---also builds upon some of the ideas of \cite{barles2016unbounded}, and thus bears some relation to our work.


\subsection{Main Results}

We begin by addressing the well-posedness of problem \eqref{vhj_whole}-\eqref{initialData_whole}.
\begin{theorem}\label{thm_Soln}
Assume $f\in \locLip(\RN)$ and $u_0\in C(\RN)$ are bounded from below. Then, there exists a unique continuous solution
of \eqref{vhj_whole}-\eqref{initialData_whole}.
\end{theorem}
The proof of Theorem \ref{thm_Soln} follows from an approximation of \eqref{vhj_whole}-\eqref{initialData_whole} by problems
set in bounded domains, together with a comparison result, Theorem \ref{thm_comparison}. We point out that the comparison
result allows us to prove the existence result only in the case when $u_0 \in \locLip(\RN)$, a case in which we have better Lipschitz
estimates since they hold up to $t=0$. 

To determine the large-time behavior of solutions to \eqref{vhj_whole}-\eqref{initialData_whole} we must introduce the corresponding
\emph{ergodic problem},
\begin{equation}\label{erg_prob_intro}
\lambda - \Delta \phi + |D\phi|^m = f(x) \quad\textrm{in } \RN,
\end{equation}
where both $\lambda$ and $\phi$ are unknown. Assuming that $m>2$ and $f\in\locLip(\RN)$ is coercive and bounded from below, it is proved in \cite{barles2016unbounded} that there a exist unique constant $\lambda^*\in\R$ and a unique $\phi\in C^2(\RN)$ (up to an additive constant) which 
together solve \eqref{erg_prob_intro} (See Theorems 2.4 and 3.1 therein). We also use key ideas and computations from \cite{barles2016unbounded} in the the proofs of Lemmas \ref{lemma_upGrowthPhi} and \ref{lemma_superlinPhi}, and in Theorem~\ref{thm_comparison}.

We now state the precise hypothesis on the right-hand side required for our main result, Theorem \ref{thm_fullConv} below
\begin{enumerate}[label=(H\arabic*)]
\item\label{growth_f} There exists an increasing function $\varphi: [0, +\infty) \to [0, +\infty)$ and constants $\alpha,  \varphi_0, f_0>0$ such that for all $r\geq 0$,
\begin{equation*}
\varphi_0^{-1} r^\alpha \leq \varphi(r)
\end{equation*}
and for all $x\in \RN$ and $r=|x|$,
\begin{equation*}
f_0^{-1} \varphi(r) -f_0 \leq f(x) \leq f_0 (\varphi(r) + 1).
\end{equation*}
\end{enumerate}

\begin{remark}
The previous assumption imposes no upper bound on the growth of $f$ at infinity, since no upper bound is assumed for the function
$\varphi$. The control on $f$ on both sides by $\varphi$ implies only that $f$ is almost radial since it ``does not vary too much for
$|x|=r$, for fixed $r>0$''. It is clear that the more growth we have on $\varphi$, the
less restrictive this condition is.
\end{remark}

Assumption \ref{growth_f} is required for the construction of the sub- and supersolutions of Section~\ref{sec_subSuper}. It also implies that the solution of the ergodic problem \eqref{erg_prob_intro} has superlinear growth at infinity (Lemma~\ref{lemma_superlinPhi}), a property which plays an important role in the proof of Theorem \ref{thm_fullConv}.


\begin{theorem}\label{thm_fullConv}
Assume \emph{\ref{growth_f}}. Then, for any $u_0\in C(\RN)$ bounded from below, there exists $\hat{c}\in\R$ such that the solution $u=u(x,t)$ of \eqref{vhj_whole}-\eqref{initialData_whole} satisfies
\begin{equation*}
u(x,t) - \lambda^*t \to \phi(x) + \hat{c}\quad\textrm{locally uniformly in }\RN\textrm{ as }t\to\infty.
\end{equation*}
\end{theorem}

The proof of Theorem \ref{thm_fullConv} follows the strategy of the corresponding result on bounded domains in \cite{tabet2010large}. However, there are a number of technical difficulties due to the fact that both the solution of \eqref{vhj_whole}-\eqref{initialData_whole} and the domain are unbounded. Heuristically, this becomes apparent in the lack the control which in the bounded-domain case is provided by the time-independent boundary conditions. This is solved by the special sub- and super solutions constructed in Lemma \ref{lemma_subSuper}, which provide the required control ``at infinity''. 

We stress that this ``control at infinity'', both in the preliminary results of Section~\ref{sec_subSuper} and in the proof of Theorem \ref{thm_fullConv}, is the main contribution of the present work, since it is achieved in spite of having no upper bound on the growth of $f$ from above, and no restriction whatsoever on the behavior of $u_0$. In contrast, the results of \cite{ichihara2012large} assume that $f$ is essentially bounded on both sides by $|x|^\alpha$ for $\alpha\geq m^*$, where $m^*=\frac{m}{m-1}$, and $u_0$ has at most polynomial growth.

\subsection{Notation}\label{subsec_Notation}

Most of the notations appearing in the text are standard. Nonnegative constants whose precise value does not affect the argument are denoted collectively by $C$, and it is usually indicated if they are taken within a given range, e.g., $0<C<1$. $B_r(x)$ denotes the ball of radius $r>0$ and center $x\in\RN$, and for simplicity we write $B_r$ for $B_r(0)$.

For $Q\subset \RN\times (0,+\infty)$, we say $Q$ is \emph{open} if and only if it is open with respect to the parabolic topology in $\mathbb{R}^{N+1}$, i.e., the topology generated by the basis $\{ B_r(x) \times (-r,t] \ |  \ x\in \RN,\ t\in (0,+\infty),\ r>0 \}$ (see e.g., \cite{wang1992regularity}). The \emph{parabolic boundary} $\partial_p Q$ of such a set is the topological boundary as defined by the parabolic topology. Similarly, $\overline{Q}$ denotes the closure of $Q$ with respect to the parabolic topology in $\mathbb{R}^{N+1}$. We still write, e.g., $\overline{B}_r$ for the closure of a subset of $\RN$ with respect to the usual topology, since there is little possibility of confusion. We also write $Q_{\hat{t}} =   \{ (x,t) \in Q \ | \ t = \hat{t} \}$ for $\hat{t}>0$ and $Q_0 =  \{(x,t) \in \overline{Q} \ |  \ t = 0 \}$.
	

\subsection{Organization of the article}

In Section \ref{sec_wellPosed_vhjwhole} we prove the existence and uniqueness of solutions to \eqref{vhj_whole}-\eqref{initialData_whole}, Theorem \ref{thm_Soln}. In Section \ref{sec_subSuper} we construct special sub- and supersolutions which will serve as comparison functions towards obtaining large-time behavior. In Section \ref{sec_LT} we prove a result on the uniform boundedness of solutions to \eqref{vhj_whole}-\eqref{initialData_whole}, Lemma \ref{lemma_unifBounded}, as well as our main result, Theorem \ref{thm_fullConv}. Finally we collect in the appendix several (already known) estimates and results which we use throughout this article.

\section{Existence and uniqueness of solutions}\label{sec_wellPosed_vhjwhole}

This section is devoted to the proof of Theorem \ref{thm_Soln}, stated in the introduction. We first show it in the case when $u_0 \in \locLip(\RN)$, the general result being obtained at the end of the section, after the proof of the comparison result.

\subsection{A first existence result}

Our first task will to obtain a solution of \eqref{vhj_whole}-\eqref{initialData_whole} via an approximation by bounded domains.
\begin{proposition}\label{thm_firstExt} 
Let $T>0$, and assume that $f, u_0 \in \locLip(\RN)$ are bounded from below. Then, there exists a continuous, bounded from below solution of
	\begin{align}
		u_t - \Delta u + |Du|^m = f(x) &{}\quad\textrm{in } \RN\times (0,T], \label{vhj_whole_T}\\
		u(x, 0) = u_0(x) &{}\quad\textrm{in } \RN.\label{initialData_whole_T} 
	\end{align}
\end{proposition}

To perform our approximation, we revisit the parabolic state-constraints problem on the approximating domains. Namely, for $R, T>0$, consider
	\begin{align}
		u_t - \Delta u + |Du|^m = f(x) &{} \quad\textrm{ in } B_R \times (0, T], \label{eqonball}\\
		u_t - \Delta u + |Du|^m \geq f(x) &{} \quad\textrm{ in } \partial B_R \times (0, T], \label{sconball}\\
		u(\cdot,0) = u_0 &{} \quad\textrm{ in } \overline{B}_R. \label{initialonball}
	\end{align}

\begin{remark}
	By standard arguments, there is no loss of generality in assuming, as in \eqref{vhj_whole_T} or \eqref{eqonball}, that the equation holds up to the terminal time $T>0$ (see e.g., \cite{evans1998partial}, Ch. 10).
\end{remark}

\begin{lemma}\label{lemma_SConBall}
Under the assumptions of Proposition \ref{thm_firstExt}, for every $R>0$ there exists a unique, continuous solution of \eqref{eqonball}-\eqref{initialonball}, denoted $u^R$.
\end{lemma}

\begin{proof}[Proof of Lemma \ref{lemma_SConBall}]
The well-posedness of \eqref{eqonball}-\eqref{initialonball} is a consequence of the strong comparison result proven in \cite{barles2004generalized}, Theorem 3.1. Existence then follows by Perron's method, provided there exist suitable sub- and supersolutions. The subsolution can be chosen as $\underline{u}(x,t) = \inf_{B_R} u_0 + t\inf_{B_R} f$ while the construction of the supersolution is more involved because of the state-constraint boundary condition.
\end{proof}

\begin{remark}\label{rmk_nonnegative}
As we mention it in the introduction, we note that both $u_0$ and $f$ may be assumed nonnegative with no loss of generality, and even larger than any constant. A fact that we will use several times.
\end{remark}

\begin{lemma}
Using the notation of Lemma \ref{lemma_SConBall}, if $R'\geq R$, then $u^{R'}(x) \leq u^R(x)$ for all $x\in \overline{B}_R$.
\end{lemma}

\begin{proof}
The result follows from the fact the solution satisfying the state-constraints boundary condition \eqref{sconball} is the maximal subsolution of
\eqref{eqonball}-\eqref{initialonball} (see e.g., \cite{crandall1992user}, Section 7.C$'$). We provide a constructive proof of the result, however, to illustrate this notion more clearly.

Observe that, for $R'\geq R$,  $u^{R'}$ is in particular a subsolution of \eqref{eqonball} in $B_R \times (0,T)$, satisfying also \eqref{initialonball}. Let $\epsilon>0$, and let $\rho_\epsilon$ denote the standard mollifier in $\mathbb{R}^{N+1}$. Since \eqref{eqonball} is convex in $(u, Du)$, $u^{R'} \ast \rho_\epsilon$ is a classical subsolution of \eqref{eqonball} in $B_R$ (this is Lemma 2.7 in \cite{barles2002convergence}). Define, for small ${\delta > 0}$, $x\in \overline{B}_R$ and $t\geq 0$,
\begin{equation*}
w(x,t) = (u^{R'} \ast \rho_\epsilon) (x,t) - \delta t.
\end{equation*}
(We have dropped the dependences of $w$ in $R'$, $\epsilon$ and $\delta$ for the sake of notational simplicity.)
We claim that $w$ is a smooth subsolution of \eqref{eqonball} in $B_R \times (0,T)$. Indeed, for all $(x,t)\in B_R \times (0,T]$, we can compute
\begin{align}
& w_t(x,t) - \Delta w(x,t) + |Dw(x,t)|^m\nonumber\\
& \quad = (u^{R'} \ast \rho_\epsilon)_t - \Delta (u^{R'} \ast \rho_\epsilon)(x,t) + |D(u^{R'} \ast \rho_\epsilon)(x,t)|^m - \delta\nonumber\\
& \quad \leq f(x) - \delta < f(x).\label{strictsubR}
\end{align}
And this computation, which is only valid for smooth enough $u^{R'}$ can be justified if $u^{R'}$ is only a continuous viscosity subsolution of \eqref{eqonball}.
 
We set $\eta_\epsilon:= \min_{\overline{B_R}}(u_0-w(x,0))$. We remark that, by the continuity of $u^{R'}$, $\eta_\epsilon\to 0$ as $\epsilon \to 0$. We now show that $u^R \geq w+ \eta_\epsilon$ in ${\overline{B_R}} \times [0,T]$ for any $T>0$. We argue by contradiction assuming that
\begin{equation*}
(u^R - w-\eta_\epsilon)(x_0,t_0) := \min_{\overline{B_R}\times[0,T]} (u^R - w-\eta_\epsilon) < 0.
\end{equation*}
By definition of $\eta_\epsilon$, we have $t_0>0$ and since $w+ \eta_\epsilon$ is smooth, the definition of the state-constraints boundary condition \eqref{sconball} implies that 
\begin{equation*}
w_t(x_0,t_0) - \Delta w(x_0,t_0) + |Dw(x_0,t_0)|^m \geq f(x_0),
\end{equation*}
which contradicts \eqref{strictsubR}. 
Therefore, $u^R\geq w+ \eta_\epsilon$ in $\overline{B_R}\times[0,T]$. Taking the limit $\epsilon, \delta \rightarrow 0$, then $T\rightarrow \infty$, we conclude.
\end{proof}

\begin{proof}[Proof of Proposition \ref{thm_firstExt}.]
We are going to obtain a solution of \eqref{vhj_whole_T}-\eqref{initialData_whole_T} as a locally uniform limit of the solutions $u^R$ of Lemma \ref{lemma_SConBall} as $R\to +\infty$.

To do so, we consider a fixed $\bar R$. For $R>2\bar R +1$, we have, assuming $u_0,f\geq 0$
$$ 0 \leq u^R \leq u^{2\bar R +1} \quad \hbox{in  }\overline{B_{2\bar R}} \times [0,T]\; ,$$
 hence $\{u^R\}_{R> 2\bar R +1}$ is uniformly bounded over $\overline{B_{2\bar R}} \times [0,T]$.
 
Furthermore, using Theorem~\ref{gradBound} and Corollary \ref{gradBoundCor} in the Appendix, the $C^{0,1/2}$-norm
of $u^R$ on $\overline{B_{\bar R}} \times [0,T]$ remains also uniformly bounded for $R>2\bar R +1$. And we also recall that
the $u^R$ are decreasing in $R$.

Thus, by using the Ascoli-Arzela Theorem together with the monotonicity of $(U^R)_R$, we have the uniform convergence of $u^R$ 
on $\overline{B_{\bar R}} \times [0,T])$. Then, by a diagonal argument, we may extract a subsequence of $\{u^R\}_{R>0}$ that converges locally uniformly to some $u\in C(\RN\times [0,T])$. By stability, it follows that $u$ is a viscosity solution of \eqref{vhj_whole_T}-\eqref{initialData_whole_T}.
\end{proof}
 

\subsection{A General Comparison Result}

We provide in this section a general comparison result, not only in $\RN\times [0, T]$, but in a more general domain $Q$ which may be a proper subset of $\RN\times [0, T]$.  The following comparison result is formulated in terms of viscosity sub- and supersolutions where $USC(Q)$ denotes the set of upper semi-continuous functions on $\overline{Q}$, while $LSC(Q)$ is the set of lower semi-continuous ones.
	
	\begin{theorem}\label{thm_comparison}
		Let $T\in (0,+\infty]$. Assume $f\in\locLip(\RN)$ is bounded from below and $Q\subset \RN\times (0,T)$ is a nonempty set, open with respect to the parabolic topology (see Subsection~\ref{subsec_Notation} for notation and definitions). Let $u\in USC(Q)$ and $v\in LSC(Q)$ be respectively a subsolution and a supersolution of Equation~\eqref{vhj_whole} in $Q$, both bounded from below. If, for all $(x,t)\in \partial_p Q$, we have $\displaystyle \limsup_{(y,s)\to (x,t)} (u-v) \leq 0$, then $u\leq v$ in $Q$.
	\end{theorem}
	
	The first main point in Theorem~\ref{thm_comparison} is that the set $Q$ is allowed to be unbounded. Consequently, $u\in USC(Q)$ and $v\in LSC(Q)$ are allowed to be unbounded as well. Naturally, we are mainly interested in the case $Q=\RN\times(0,+\infty)$. However\textemdash and this is why we need such a general formulation\textemdash in the study of the large time behavior, we have to use such comparison result with a supersolution $V$ which is not defined in the whole space but on a set $Q$ on which we know that $V(x,t)\to+\infty$ if $(x,t)\to \partial_p Q$ with $t>0$.

\begin{proof}
To deal with the difficulty of $u$ and $v$ being unbounded, we define
\begin{equation*}
z_1(x,t) = -e^{-u(x,t)}, \qquad z_2(x,t) = -e^{-v(x,t)}.
\end{equation*}
Since $u$ and $v$ are bounded from below, $z_1$ and $z_2$ are bounded. Furthermore, we have that $z_1$ and $z_2$ are respectively a sub- and supersolution of
\begin{equation}\label{eq_transformed}
z_t - \Delta z + N(x,z,Dz) = 0,
\end{equation}
where $N: \RN \times \R \times \RN\to \RN$ is given by
\begin{equation}\label{rhs_transformed}
N(x,r,p) = r\left(f(x) + \left|\frac{p}{r}\right|^2 - \left|\frac{p}{r}\right|^m\right).
\end{equation}
Formally, this is a straightforward computation. Given the monotonicity of the transformation defining $z_1, z_2,$ there is no difficulty in passing the computation over to smooth test functions. Another consequence is that $z_1(x,t)\leq z_2(x,t)$ if and only if $u(x,t)\leq v(x,t)$. Hence Theorem \ref{thm_comparison} follows from proving the same comparison result for sub- and supersolutions of \eqref{eq_transformed}.

We proceed with the usual scheme of doubling variables, and for simplicity we first treat the case in which $Q$ is unbounded in $t$. Assume that the conclusion of the theorem is false and $M:= \sup_Q (z_1 - z_2)> 0$. Define first, for $\delta>0$,
	\begin{equation}\label{M_pen1}
	M_\delta := \sup_Q \left( z_1(x,t) - z_2(x,t) - \delta(|x|^2 + t)\right) .
	\end{equation}
As the penalized function on the right-hand side of \eqref{M_pen1} is upper-semicontinuous and goes to $-\infty$ as either $|x|\to+\infty$ or $t\to +\infty$, the supremum in achieved at some $(\bar{x}, \bar{t})\in \overline{Q}$. From standard arguments, we have that $M_\delta\to M$ as $\delta\to0$ (see e.g., \cite{crandall1992user}, Proposition 3.7), which implies that for small enough $\delta$, $M_\delta>0$. Moreover, since $\limsup_{(y,s) \to (x,t)} (z_1(y,s) - z_2(y,s))$ for all $(x,t)\in \partial_p Q$, we necessarily have $(\bar{x}, \bar{t})\in Q$. Define now
	\begin{equation*}
	M_{\delta, \alpha} := \sup_{Q\times Q} \left\{ z_1(x,t) - z_2(y,t) - \frac{\delta}{2}(|x|^2 + |y|^2 + t) - \frac{\alpha}{2}|x-y|^2 \right\}.
	\end{equation*}
Again the supremum above if achieved at a point $(\hat{x}, \hat{y}, \hat{t})$. Similarly, we know that as $\alpha\rightarrow\infty$ and $\delta$ remains fixed,  $M_{\alpha, \delta} \rightarrow M_\delta$ and $\hat{x}, \hat{y}\to\bar{x}$ and $\hat{t}\to \bar{t}$ for $\bar{x}, \bar{t}$ as above.

Thus, an application of Ishii's Lemma (see \cite{crandall1992user}, Theorem 3.2) gives
\begin{equation}\label{variationineq}
N(\hat{x}, z_1(\hat{x}, \hat{t}), \alpha(\hat{x}- \hat{y}) + \delta \hat{x}) - N(\hat{y}, z_2(\hat{y}, \hat{t}), \alpha(\hat{x}- \hat{y}) - \delta \hat{y}) \leq \left(2n + \frac{1}{2}\right) \delta.
\end{equation}

We aim to bound this difference from below by a positive constant independent of $\delta$. Let
\begin{equation*}
h(s) = N(a(s), b(s), c(s)),
\end{equation*}
where
\begin{align*}
& a(s) = s\hat{x} + (1-s)\hat{y},\qquad b(s) = sz_1(\hat{x}, \hat{t}) + (1-s)z_2(\hat{y}, \hat{t}),\\
& c(s) = \alpha (\hat{x} - \hat{y}) + \delta(s\hat{x} + (s-1)\hat{y}).
\end{align*}
We rewrite \eqref{variationineq} as
\begin{align}\label{variationineq_ReWr}
& N(\hat{x}, z_1(\hat{x}, \hat{t}), \alpha(\hat{x} - \hat{y}) + \delta \hat{x}) - N(\hat{y}, z_2(\hat{y}, \hat{t}), \alpha(\hat{x}- \hat{y}) - \delta \hat{y})\nonumber\\
&{} = N(a(1), b(1), c(1)) -  N(a(0), b(0), c(0))\nonumber\\
&{} = h(1) - h(0) = \int_0^1 h'(s) \,ds.
\end{align}
Computing
\begin{align*}
& a'(s) = \hat{x} - \hat{y},\qquad b'(s) = z_1(\hat{x}, \hat{t}) - z_2(\hat{y}, \hat{t}),\\
& c'(s) = \delta (\hat{x} + \hat{y}),
\end{align*}
and, from \eqref{rhs_transformed},
\begin{align*}
& \frac{\partial N}{\partial x} = r\,Df(x),\qquad \frac{\partial N}{\partial r} = f(x) - \left|\frac{p}{r}\right|^2 +(m-1) \left|\frac{p}{r}\right|^m,\\[4pt]
& \frac{\partial N}{\partial p} = \left( \frac{2}{|r|^2} - \frac{m(m-2) |p|^{m-2}}{|r|^m}\right) p,
\end{align*}
we have
\begin{align*}
\int_0^1 h'(s) \,ds &{}= \int_0^1 \frac{\partial N}{\partial x}(a(s), b(s), c(s)) \cdot a'(s) + \frac{\partial N}{\partial r}(a(s), b(s), c(s)) b'(s) \\
&\qquad{}+ \frac{\partial N}{\partial p}(a(s), b(s), c(s)) \cdot c'(s) \,ds.
\end{align*}
We proceed to bound each term in the last expression
\begin{align*}
\frac{\partial N}{\partial x}(a(s), b(s), c(s)) \cdot a'(s) \geq -\max\{\|z_1\|_\infty, \|z_2\|_\infty\}\|Df\|_{\infty,B_{R(\delta)}}|\hat{x}-\hat{y}|,
\end{align*}
where $B_{R(\delta)}$ is a ball that contains $(\hat{x}, \hat{y})$, the points at which the maximum of $\Phi=\Phi^{\delta,\alpha}(x,y)$ is achieved, considering $\delta>0$ is fixed; $B_{R(\delta)}$ is uniform in $\alpha$. Since $\hat{x}-\hat{y}\rightarrow 0$ as $\alpha\rightarrow\infty$, this implies
\begin{equation*}
\frac{\partial N}{\partial x}(a(s), b(s), c(s)) \cdot a'(s) \rightarrow 0 \quad\textrm{as } \alpha \rightarrow \infty.
\end{equation*}
We write $q:=\frac{p}{r}$. Young's inequality gives
\begin{equation*}
|q|^2 \leq \frac{2}{m}|q|^m + \frac{m}{m-2},
\end{equation*}
hence
\begin{align}\label{partial_NrBound}
\frac{\partial N}{\partial r} &{} = f(x) - \left|\frac{p}{r}\right|^2 +(m-1) \left|\frac{p}{r}\right|^m\nonumber\\
&{}\geq f(x) + (m-1)|q|^m - \frac{2}{m}|q|^m - \frac{m}{m-2}\nonumber\\
&{}\geq f(x) - \frac{m}{m-2} + \left( m - 1 - \frac{2}{m} \right)|q|^m\nonumber\\
&{}\geq 1 + C |q|^m,
\end{align} 
taking $0< C < m - 1 - \frac{2}{m}$ in the last inequality. The bound 
\begin{equation}\label{big_f_assumed}
f(x) \geq 1 + \frac{m}{m-2} \quad\textrm{for all } x\in\RN,
\end{equation}
may be assumed without loss of generality by initially considering, instead of $u$ and $v$, the functions
\begin{equation*}
u(x,t) + \left(1 + \frac{m}{m-2}\right) t, \qquad v(x,t) + \left(1 + \frac{m}{m-2}\right) t,
\end{equation*}
for all $(x,t)\in Q$ (see the comments at the beginning of the introduction). On the other hand, by direct computation,
\begin{align*}
\left|\frac{\partial N}{\partial p}\right| &{}= 2\left|\frac{p}{r}\right| + m\left|\frac{p}{r}\right|^{m-1} = 2|q|+ m|q|^{m-1}\\
&{}\leq m(|q| +|q|^{m-1}) \leq 2m(1 + |q|^m).
\end{align*}
Here we've used that $m>2$ and that $|q| +|q|^{m-1}\leq 2(1 + |q|^m)$, which follows easily by considering the cases $|q|>1$ and $|q|\leq1$ separately. We have therefore obtained
\begin{equation}\label{dompartial}
\left(1 + \frac{2m}{C}\right)\frac{\partial N}{\partial r} \geq \left|\frac{\partial N}{\partial p}\right|.
\end{equation}
Thus, using \eqref{partial_NrBound} and \eqref{dompartial}, for small $\delta>0$ we have
\begin{align}
\liminf_{\alpha\rightarrow \infty} \int_{0}^{1} h'(s) \,ds &{}\geq \liminf_{\alpha\rightarrow \infty} \int_{0}^{1} \frac{\partial N}{\partial r} (z_1(\hat{x}, \hat{t}) - z_2(\hat{y}, \hat{t})) + \frac{\partial N}{\partial p} \cdot \delta(\hat{x} + \hat{y}) \,ds\nonumber\\
&{}\geq \liminf_{\alpha\rightarrow \infty} \int_{0}^{1} \frac{\partial N}{\partial r} \left(M_{\alpha,\delta} + \delta(|\hat{x}|^2 + |\hat{y}|^2 + \hat{t}) + \frac{\alpha}{2}|\hat{x} - \hat{y}|^2\right)\nonumber\\
&{}\quad + \frac{\partial N}{\partial p} \cdot \delta(\hat{x} + \hat{y}) \,ds\nonumber\\
&{}\geq \liminf_{\alpha\rightarrow \infty} \int_{0}^{1} \frac{\partial N}{\partial r} M_{\alpha,\delta} - \delta\left|\frac{\partial N}{\partial p}\right|(|\hat{x}| + |\hat{y}|) \,ds\nonumber\\
&{}\geq \int_{0}^{1} \frac{\partial N}{\partial r} \left(M_\delta - 2\delta\left(1 + \frac{2m}{C}\right)|\bar{x}|\right) \,ds\nonumber\\
&{}> \frac{M}{2} >0.\label{contradicIneq}
\end{align}
Here we have also used the facts concerning the limits $\alpha\to\infty$, $\delta\to0$ (taken in that order) mentioned at the outset of the argument. Together with \eqref{variationineq} and \eqref{variationineq_ReWr}, \eqref{contradicIneq} gives the desired contradiction.

Finally, if $Q\subset \RN \times [0,T]$ for some $T>0$, we note that the preceding argument equally applies, but the penalization term in $t$ is no longer necessary in \eqref{M_pen1}. That is, we may define
		\begin{equation*}
			M_\delta := \sup_Q \left( z_1(x,t) - z_2(x,t) - \delta |x|^2\right),
		\end{equation*}
	instead of \eqref{M_pen1} and proceed in the same way.
\end{proof}

\begin{remark} In \cite{barles2016unbounded}, an analogous comparison result is proved in the complementary of ball $B_R$ if $R$ is large enough (depending on the coercivity for $f$): here, because of the parabolic framework, such argument can be performed in the whole space.
\end{remark}


\subsection{Conclusion}

\begin{proof}[Proof of Theorem \ref{thm_Soln}]
We note that, from Remark \ref{rmk_nonnegative}, any solution of \eqref{vhj_whole_T}-\eqref{initialData_whole_T} is bounded from below over $\RN\times [0,T]$, since it is nonnegative. Therefore, combining the results of Proposition \ref{thm_firstExt} and Theorem \ref{thm_comparison}, there exists a unique solution of \eqref{vhj_whole_T}-\eqref{initialData_whole_T}. Since $T>0$ is arbitrary (in particular, with no dependence on the data) the solution can be uniquely extended to a solution of \eqref{vhj_whole}-\eqref{initialData_whole}, which is also nonnegative. Uniqueness then follows immediately from Theorem~\ref{thm_comparison}. And the case when $u_0 \in \locLip(\RN)$ is complete.

When $u_0 \in C(\RN)$, we argue in the following way: by classical results, there exists a sequence $(u^\epsilon_0)_\epsilon$ of functions of $\locLip(\RN)$ such 
that, for any $\epsilon$, $ |u^\epsilon_0(x)-u_0(x)|\leq \epsilon$ in $\RN$. If $u^\epsilon$ denotes the unique solution of \eqref{vhj_whole} associated to the initial data $u^\epsilon_0$, we have, by the comparison result
$$ |u^\epsilon (x,t) -u^{\epsilon'}(x,t)|\leq  \epsilon' + \epsilon \quad \textrm{in } \RN\times [0,+\infty)\; ,$$
since $|u^\epsilon_0 (x) -u^{\epsilon'}_0(x)|\leq  \epsilon' + \epsilon$ in $\RN$.
Therefore, the Cauchy sequence $(u^\epsilon)_\epsilon$ converge uniformly in $ \RN\times [0,+\infty)$ to the unique viscosity solution of of \eqref{vhj_whole}-\eqref{initialData_whole} by Theorem~\ref{thm_comparison}.
\end{proof}

\section{Sub- and supersolutions}\label{sec_subSuper}

In this section we construct sub- and supersolutions to a modified evolution problem. This will allow us to relate the solution of \eqref{vhj_whole}-\eqref{initialData_whole} to the solution of the associated ergodic problem \eqref{erg_prob_intro}. To this end, we first obtain some preliminary results regarding the solution of \eqref{erg_prob_intro}.

\subsection{Estimates for the solution of the ergodic problem}\label{subsec_ErgEstimates}

\begin{lemma}\label{lemma_upGrowthPhi} Assume that $f \in C(\RN)$ and set $M_R = \sup_{B_R} |f - \lambda^*|$. Then, for any $R>1$, there exists a constant $C_1 >0$ such that any solution $\phi$ of \eqref{erg_prob_intro} satisfies
\begin{equation*}
|\phi(x)-\phi(0)| \leq C_1 R M_R^{\frac1m}\quad\textrm{for all } x\in \RN,\ |x|\leq R.
\end{equation*}
In particular, if \emph{\ref{growth_f}}, for any $R>1$ and for any solution $\phi$ of \eqref{erg_prob_intro}, there exists a constant $C>0$ such that
\begin{equation*}
\phi(x) -\phi(0) \leq C\left(\varphi(R) ^{\frac{\alpha + m}{\alpha m}} + 1 \right) \quad\textrm{for all } x\in \RN,\ |x|\leq R.
\end{equation*}
\end{lemma}
This result provides a surprising estimate on  $|\phi(x)-\phi(0)|$ since this is essentially the same as if $\phi$
were solving the first-order PDE $|D\phi(x)|^m = f (x) - \lambda^*$ and is a concrete evidence that the Laplacian term just improves the estimate. But two other points may be not be so natural: first we have such estimate on $\phi$ in $B_R$ by using a bound on $f$ on the same ball $B_R$, contrarily to a Bernstein estimate which would use a bound on $f$ on $B_{R+1}$. Finally, and this is not the least surprising point, this estimate is based on the {\em H\"older estimates} of Capuzzo Dolcetta, Leoni and Porretta \cite{capuzzo2010holder} which is recalled in the appendix (see Theorem~\ref{thm_holder}). We point out that the same result also holds if $f$ is in $L^q$ for $q >N/m$ by results of 
Dall’Aglio and Porretta \cite{dall2015local}.

\begin{proof}
Let $\gamma= \frac{m-2}{m-1}$, and for $R>0$, define
\begin{equation*}
v(x) = R^{-\gamma} \phi(Rx).
\end{equation*}
Since $\phi$ is, in particular, a solution of \eqref{erg_prob_intro} in $B_R$, $v$ is a solution of
\begin{equation*}
-\Delta v + |Dv|^m = R^{2-\gamma} (f(Rx) - \lambda^*) \leq R^{2-\gamma} M_R \quad \textrm{in } B_1,
\end{equation*}
 Now set $K_R = R^{2-\gamma} M_R$ and
\begin{equation*}
u(x) = K_R^{-\frac{1}{m}} v(x)
\end{equation*}
to obtain that $u$ is a solution of
\begin{equation}\label{rescaled_erg}
-K_R ^{\frac{1}{m} - 1} \Delta u + |Du|^m \leq 1 \quad\textrm{in } B_1.
\end{equation}
Observe that $m>2$ implies that $\gamma < 2$ and $\frac{1}{m} - 1<0$. Since $f$ is (in particular) coercive, we have that $K_R\to 0$ as $R\to \infty$. Hence, for $R$ sufficiently large we obtain $0 \leq K_R \leq 1$. 

We are now in position to apply the estimate of Theorem \eqref{thm_holder} (in the appendix) to Equation \eqref{rescaled_erg}: there exists a constant $C_1>0$, \emph{independent} of $R>0$, such that
\begin{equation*}
|u(x) - u(y)| \leq C_1 |x-y|^\gamma\quad \textrm{for all } x,y\in \overline{B}_1. 
\end{equation*}
In terms of $\phi$, for $|x|\leq 1, y=0$ this gives
\begin{equation*}
K_R^{-\frac{1}{m}}R^{-\gamma} |\phi(Rx) - \phi(0)| \leq C_1,
\end{equation*}
or, for $|x| \leq R$,
\begin{equation*}
|\phi(x) - \phi(0)|  \leq C_1 K_R^{\frac{1}{m}}R^\gamma = C_1 (R^{2-\gamma} M_R)^{\frac{1}{m}} R^{-\gamma} = C_1M_R^{\frac{1}{m}} R\; .
\end{equation*}

If (H1) holds, by taking, if necessary, a larger $C$ depending on $\phi(0)$, $f_0$ and $\lambda^*$, we conclude 
\begin{equation}\label{intermediate_UpPhi}
\phi(x) \leq C\varphi(R)^{\frac{1}{m}} R \leq C\varphi(R)^{\frac{1}{m}} \varphi(R)^{\frac{1}{\alpha}} =  C\varphi(R)^{\frac{\alpha  + m}{\alpha m}}.
 \end{equation}
\end{proof}

\begin{lemma}\label{lemma_superlinPhi}
Assume \emph{\ref{growth_f}}. Then, the solution $\phi$ of \eqref{erg_prob_intro} satisfies
\begin{equation*}
\lim\limits_{x\to +\infty} \frac{\phi(x)}{|x|} = +\infty.
\end{equation*}
\end{lemma}

\begin{proof} 
In order to prove Lemma~\ref{lemma_superlinPhi}, we may assume without loss of generality that $\phi\geq 0$ in $\RN$ since $\phi$ is bounded from above and we can change $\phi$ in $\phi+C$ for some large constant $C>0$.

We argue by contradiction. Assume that there exists a sequence $(y_n)_{n\in \mathbb{N}}$ with $y_n \to +\infty$ as $n\to +\infty$ such that 
\begin{equation}\label{phi_low_contrad}
\limsup\limits_{n\to +\infty} \frac{\phi(y_n)}{|y_n|} = M < +\infty
\end{equation}
for some $M\geq 0$, and define for $y \in B_1$,
\begin{equation*}
v_n(y) = \frac{\phi(y_n + \frac{|y_n|}{2} y)}{|y_n| \varphi(\frac{1}{2}|y_n|)^{\frac{1}{m}}}.
\end{equation*}
The function $v_n$ satisfies
\begin{align*}
&{}-2^{2-m} |y_n|^{-1} \varphi(\frac{1}{2}|y_n|)^{\frac{1}{m} - 1} \Delta v_n (y) \, + \, |Dv_n|^m\\
&{}\quad = 2^{-m}\varphi(\frac{1}{2}|y_n|)^{-1}\left(f\left(y_n + \frac{|y_n|}{2}y\right) - \lambda^*\right)\\
&{}\quad \geq 2^{-m}\varphi(\frac{1}{2}|y_n|)^{-1}\left(f_0^{-1} \varphi(\frac{1}{2}|y_n|) - \lambda^*\right),
\end{align*}
where we have used \ref{growth_f} for the last inequality together with the fact that $|y_n + \frac{|y_n|}{2}y| \geq |y_n| - \frac{|y_n|}{2}|y| = \frac{1}{2}|y_n|$. Hence, if $\epsilon_n= 2^{2-m} |y_n|^{-1} \varphi(\frac{1}{2}|y_n|)^{\frac{1}{m} - 1}$, we have $\epsilon_n\to 0$ since $\varphi$ is coercive and $1 - \frac{1}{m} < 0$ and
\begin{equation}\label{vn-eqn}
-\epsilon_n \Delta v_n (y) \, + \, |Dv_n|^m \geq 2^{-m} f_0^{-1} -o_n(1) \quad \hbox{in  }B_1\; ,
\end{equation}
where $o_n(1)\to 0$ as $n\to \infty$.

In order to pass to the limit, we lack some $L^\infty$-bounf on $v_n$. To overcome this difficulty, we set
\begin{equation*}
\tilde{v}_n (y) = \min\left (v_n (y), K(1-|y|)\right) \quad \hbox{in  }B_1,
\end{equation*}
for some $K\geq 2^{-m} f_0^{-1}+1$. For $n$ large enough, the concave function $y \mapsto K(1-|y|)$ is also a supersolution of \eqref{vn-eqn} and therefore so is $ \tilde  v_n$ as the minimum of two supersolutions.

Hence, the half-relaxed limit $\tilde v = \liminf\limits_{n\to\infty}\,\!\!^* \,\tilde v_n$ is well defined and, by stability, we have in the limit $n\to +\infty$ that
\begin{equation*}
|D\tilde v|^m \geq 2^{-m} f_0^{-1} \quad \textrm{in } B_1
\end{equation*}
in the viscosity sense. For all $n$, $\tilde v_n \geq 0$, hence also $\tilde v\geq 0$. Thus $\tilde v$ is a supersolution of the eikonal equation $|Du|^m = 2^{-m} f_0^{-1}$ with homogeneous boundary condition since $\tilde v_n = 0$ on $\partial B_1$. The latter has the unique solution $\frac{1}{2}f_0^{-\frac{1}{m}}d(y, \partial B_1)$. Therefore, by comparison, $\tilde v(y) \geq \frac{1}{2}f_0^{-\frac{1}{m}}d(y, \partial B_1)$. However, using \eqref{phi_low_contrad} we have
\begin{equation*}
0 \leq \tilde v(0) \leq  \liminf\limits_{n\to +\infty} \frac{\phi(y_n)}{|y_n| \varphi(\frac{1}{2}|y_n|)} \leq \limsup\limits_{n\to +\infty} \frac{\phi(y_n)}{|y_n|}\cdot \lim\limits_{n\to +\infty} \frac{1}{\varphi(\frac{1}{2}|y_n|)^\frac{1}{m}} = M\cdot 0 = 0,
\end{equation*}
and this implies $0 = \tilde v(0) \geq \frac{1}{2}f_0^{-\frac{1}{m}}d(0, \partial B_1) = \frac{1}{2}f_0^{-\frac{1}{m}} > 0$, a contradiction.
\end{proof}

\begin{remark}
Lemmas \ref{lemma_upGrowthPhi} and \ref{lemma_superlinPhi} correspond to Propositions 3.3 and 3.4 in \cite{barles2016unbounded}, where the polynomial growth rates which are assumed for $f$ lead to polynomial rates for the solution $\phi$. In particular, the proof of Lemma \ref{lemma_superlinPhi} is based on the same idea as that of Proposition 3.4 in \cite{barles2016unbounded}.
\end{remark}

\subsection{Construction of sub- and supersolutions}\label{subSec_subSuper}
The proof of the asymptotic behavior of $u$ is done in two main steps: the first one consists in showing
that there exist two constants $C_1, C_2>0$ such that
\begin{equation*}
\phi(x) -C_1 \leq \liminf_{t\to +\infty} \left(u(x,t)  - \lambda^*t \right) \leq \limsup_{t\to +\infty} (u(x,t)  - \lambda^*t)  \leq \phi + C_2\quad\hbox{in  }\RN,
\end{equation*}
where $(\lambda^*, \phi)$ is a solution pair of \eqref{erg_prob_intro}. Then, in the second one, we show that this property implies the convergence.

The aim of this section is to build suitable sub- and supersolutions to perform the first step. We point out that we
face here the difficulty of the transition from $u(\cdot,0)= u_0$\textemdash which may have any growth at infinity\textemdash to $u(x,t)$, which looks like $\phi$ for large time.

\begin{lemma}\label{lemma_subSuper}
Assume \emph{\ref{growth_f}} and let $\phi$ be any solution of \eqref{erg_prob_intro}. Then, there exist an open, nonempty set $Q\subset \RN\times (0,+\infty)$, and functions $U\in USC(\RN\times(0,T))$
and $V\in LSC(Q)$ which are bounded from below sub- and supersolution of
	\begin{equation}\label{mod_ParabErg}
		w_t - \Delta w + |Dw|^m = f(x) - \lambda^*
	\end{equation}
in $\RN\times (0,+\infty)$ and in $Q$ respectively. Furthermore, $U$ and $V$, together with $Q$, satisfy the following:
	\begin{enumerate}[label=(\roman*)]
		\item\label{Q_props} For any compact $K\subset \RN$, there exists a $\hat{t}\geq 0$ such that $K\subset Q_t$ for all $t\geq \hat{t}$.
		\item\label{subSuper_locUnifConv} There exist constants $\sigma_1, \sigma_2>0$ such that $U(\cdot,t)\to\phi - \sigma_1$ and $V(\cdot, t) \to \phi + \sigma_2$ locally uniformly in $\RN$ as $t\to+\infty$.
		\item\label{super_Blows} If either $(x,t) \in \partial_p Q$ with $t>0$, or $x\in \partial Q_0$ and $t=0$, we have
			\begin{equation*}
				V(y,s) \to +\infty \quad\textrm{as }(y,s) \to (x,t),\ (y,s)\in Q\; . 
			\end{equation*}
		\item\label{sub_Bounded} There exists $M>0$ such that, for all $t>0$,
			\begin{equation*}
			U(x,t) \leq t + M \quad\textrm{for all } x\in\RN.
			\end{equation*}
	\end{enumerate}
\end{lemma}

\begin{remark}
	The importance of Property \ref{super_Blows} of the lemma is that it allows us to construct a supersolution of \eqref{vhj_whole}-\eqref{initialData_whole} with the desired properties without assuming any restriction on the growth at infinity for the initial data.
\end{remark}

\begin{proof}
We begin with the construction of the supersolution $V$, and will later indicate the necessary changes to obtain $U$. While the constructions are similar, it is not the case that one can be obtained from the other.\\

\noindent\textit{Construction of the supersolution.} We first define
	\begin{equation*}
	v(x,t) = \phi(x) + \chi(\phi(x) - R(t)),
	\end{equation*}
where $R:\R\to\R$ and $\chi:(-\infty, b)\to\R$ are smooth functions to be chosen later on, as is the endpoint $b>0$.

In order to make this choice, we assume that $\chi' \geq 0$ in $(-\infty, b)$ and $R(t)\to +\infty$ as $t\to +\infty$; Property \textit{\ref{subSuper_locUnifConv}} of the lemma suggests to set 
	\begin{equation}\label{chi_flat}
	\chi(s) \equiv 0 \quad  \textrm{for all }s\leq 0.
	\end{equation}

The set on which the supersolution will be obtained is $Q=\{(x,t)\in \RN\times (0,+\infty) \ | \ \phi(x) - R(t) < b \}$. We will check that $Q$ has the required properties once the choice of $R(t)$ is made. Also, to have a suitable behavior of the supersolution on $\partial_p Q$ in order to have Property \textit{\ref{super_Blows}}, we require that 
	\begin{equation}\label{chiBlows}
	\chi(s)\to +\infty\quad\textrm{as } s\to b^-.
	\end{equation}

To continue to identify the required properties on $\chi$ we perform a preliminary computation, in which the argument $\phi(x) - R(t)$ in the derivatives of $\chi$ will be omitted to have simpler notations. Using \eqref{erg_prob_intro}, we have
	\begin{align}\label{supercomp}
		&v_t- \Delta v + |Dv|^m - f(x) + \lambda^*\nonumber\\
		&{}\quad= -\chi'\dot{R}(t) - (1 + \chi')(\lambda^* - f(x) + |D\phi|^m) - \chi''|D\phi|^2\nonumber\\
		&{}\qquad+ (1 + \chi')^m|D\phi|^m - f(x) + \lambda^*\nonumber\\
		&{}\quad= \chi' (-\dot{R}(t) + f(x) - \lambda^*) + \left[(1 + \chi')^m - (1 + \chi')\right]|D\phi|^m - \chi''|D\phi|^2\nonumber\\
		&{}\quad\geq \chi' (-\dot{R}(t) + f(x) - \lambda^*) + (1-\nicefrac{2}{m})\left[(1 + \chi')^m - (1 + \chi')\right]|D\phi|^m\nonumber\\
		&{}\qquad\ - \frac{m-2}{m}\left(\frac{\chi''}{[(1 + \chi')^m - (1 + \chi')]^\frac{2}{m}}\right)^\frac{m}{m-2}\nonumber\\
		&{}\quad\geq \chi' (-\dot{R}(t) + f(x) - \lambda^*) - \frac{m-2}{m}\left(\frac{\chi''}{[(1 + \chi')^m - (1 + \chi')]^\frac{2}{m}}\right)^\frac{m}{m-2}.
	\end{align}
Here we have used Young's inequality and $m>2$ to control the term containing $|D\phi|^2$ and dropped the resulting nonnegative expression in $|D\phi|^m$.

At this point it is perhaps convenient to outline our main argument: it follows from \eqref{chi_flat} and \eqref{supercomp} that if $\phi(x) - R(t) \leq 0$, then $v$ as defined at the outset is trivially a supersolution of \eqref{mod_ParabErg}. Thus, the crucial point of the proof is to show that if $\phi(x) - R(t) > 0$, then $f(x)$ is necessarily \emph{large}. The precise size of $f$ which is needed can be precisely quantified through \ref{growth_f} and Lemma~\ref{lemma_upGrowthPhi}, and this allows us to bound the right-hand side of \eqref{supercomp} from below.

To this end, we have to control more precisely the term containing $\chi''$ in the last line of \eqref{supercomp} by choosing $\chi$ in a right way. We want to have
\begin{equation}\label{PhiOfChi}
\frac{m-2}{m}\left(\frac{\chi''}{[(1 + \chi')^m - (1 + \chi')]^\frac{2}{m}}\right)^\frac{m}{m-2} \leq \Phi(\chi')
\end{equation}
for a suitable function $\Phi$ which gives a uniform bound from below for the expression
\begin{equation*}
\chi' (-\dot{R}(t) + f(x) - \lambda^*) - \Phi(\chi').
\end{equation*}
From \eqref{chi_flat} and \eqref{chiBlows}, we know $\chi'$ takes values from $0$ to $+\infty$, thus $\Phi$ must have at most linear growth. The correct choice for estimate \eqref{PhiOfChi} turns out to be $\Phi(s)=s^\beta$, where $\beta\in (0,1)$ can be taken arbitrarily close to $1$. We state and prove the existence of such $\chi$ in Proposition~\ref{prop_chi} at the end of this section in order to continue with our main argument.

Continuing from \eqref{supercomp} and using Proposition \ref{prop_chi}, we now have
\begin{equation}\label{supercomp2}
v_t- \Delta v + |Dv|^m - f(x) + \lambda^* \geq \chi' (-\dot{R}(t) + f(x) - \lambda^*) - (\chi')^\beta.
\end{equation}
From the previous considerations, we assume henceforth that $\phi(x)>R(t)$. Using Lemma \ref{lemma_upGrowthPhi}, we have,
\begin{equation*}
R(t) < \phi(x) \leq C\left(\varphi(|x|) ^{\frac{\alpha + m}{\alpha m}} +1 \right),
\end{equation*}
and from \ref{growth_f},
\begin{equation}\label{fdepOn_t_1}
f(x) \geq f_0^{-1}\left[\left(\frac{R(t) - C}{C}\right)^+\right]^{\frac{\alpha m}{\alpha + m}} - f_0.
\end{equation}
(In fact, reasoning as in Remark \ref{rmk_nonnegative}, the term $-f_0$ on the right may be omitted, but this is of little relevance to the computation.)
Thus, if $R(t)$ is chosen large compared to $\dot{R}(t)$, we have $ -\dot{R}(t) + f(x) - \lambda^* > 0$, and may use Young's inequality again to estimate
\begin{align}\label{supercomp3}
&{}v_t- \Delta v + |Dv|^m - f(x) + \lambda^*\nonumber\\
&{}\quad\geq (1-\beta) \left(\chi'(-\dot{R}(t) + f(x) - \lambda^*) - (f(x) - \lambda^* + 1)^{-\frac{\beta}{1-\beta}}\right)\nonumber\\
&{}\quad\geq -(1-\beta)(-\dot{R}(t) + f(x) - \lambda^*)^{-\frac{\beta}{1-\beta}}.
\end{align}
Furthermore, \eqref{fdepOn_t_1} implies that for some $0<\hat{\alpha}<\frac{\alpha m}{\alpha + m}$ and large enough $R(t)$ we have
\begin{equation*}
R(t)^{\hat{\alpha}} + 1 \leq -\dot{R}(t) + f(x) - \lambda^*.
\end{equation*}
Consequently,
\begin{equation}\label{big_f2}
(-\dot{R}(t) + f(x) - \lambda^*)^{-\frac{\beta}{1-\beta}} \leq (R(t)^{\hat{\alpha}} + 1)^{-\frac{\beta}{1-\beta}} \quad\textrm{for all } t>0.
\end{equation}
We write $\hat{\beta}=\frac{\beta}{1-\beta}$, and let
\begin{equation*}
\psi(t) = (1-\beta)\int_0^t (R(\tau)^{\hat{\alpha}} + 1)^{-\hat{\beta}} \,d\tau.
\end{equation*}

At this point we set $R(t) = t + t_0$. Thus, the set $Q= \{\phi(x) - (t+t_0) < b\}$\textemdash as defined at the beginning of the proof\textemdash is open and satisfies Property \textit{\ref{Q_props}} of the lemma. In fact, it is easy to see by using the continuity of $\phi$ that given a compact set $K\subset \RN$, $t_0>0$ can chosen large enough so that $K\subset Q_0$, and by construction, it is clear that if $(x,t)\in Q$ then $(y,s)\in Q$ for all $s\geq t$. We remark that Property~\textit{\ref{Q_props}} is required to take the limits in Property~\textit{\ref{subSuper_locUnifConv}}.

With respect to the previous computations, since $\dot{R}(t)=1$, $R(t)$ can be taken large while $\dot{R}(t)$ remains bounded, as was assumed. Furthermore, since $\hat{\beta}$ can be made arbitrarily large by taking $\beta$ close to 1 and $\hat{\alpha}>0$, the integral defining $\psi$ remains bounded as $t\to+\infty$. We write for further reference 
	\begin{equation}\label{c2Def}
	\sigma := (1-\beta)\int_0^{+\infty}((\tau + t_0)^{\hat{\alpha}} + 1)^{-\hat{\beta}} \,d\tau < +\infty.
	\end{equation}
Note also that $\psi$ is smooth and $\psi'(t) = (1-\beta)((t+t_0)^{\hat{\alpha}} + 1)^{-\hat{\beta}}>0$ for all $t>0$. 

To summarize, defining
	\begin{equation*}
	V(x,t) = v(x,t) + \psi(t),
	\end{equation*}
with $v$ and $\psi$ as above, yields, by \eqref{supercomp3} and \eqref{big_f2}, that
	\begin{equation*}
	\partial_t V - \Delta V + |DV|^m - f(x) + \lambda^* \geq 0 \quad\textrm{for all } (x,t) \in Q\; .
	\end{equation*}
Finally, since $\phi \in C^2(\RN)$ is bounded from below and both $\chi$ and $\psi$ are smooth and nonnegative, $V\in C^{2,1}(Q)$ is bounded from below.

It is easy to check that $V$ satisfies parts \ref{subSuper_locUnifConv} and \ref{super_Blows} of the lemma. Let $K\subset \RN$ be a compact set. Then, for large enough $t>0$, so that $t >\phi(x) - t_0$ for all $x\in K$, we have $\chi(\phi(x) - (t + t_0))=0$ by \eqref{chi_flat}. Thus, as $t\to+\infty$,
	\begin{equation*}
		V(x,t) = \phi(x) + (1-\beta)\int_0^t (\tau^{\hat{\alpha}} + 1)^{-\hat{\beta}} \,d\tau\to\phi(x) + \sigma,
	\end{equation*}
uniformly over $K$. On the other hand, for fixed $(x,t)\in \partial_p Q$ and any $(y,s)\in Q$, the definition of $Q$ implies that $\phi(y) - (s + t_0)$ approaches $b$ from below as $(y,s)\to (x,t)\in \partial_p Q$, and similarly for $(y,s)\to (x,0)$ if $x\in \partial Q_0$. This implies Property~\ref{super_Blows}.\\

\noindent\textit{Construction of the subsolution.} As mentioned earlier, the construction of the subsolution $U$ is analogous. We briefly go over the main points. Define
\begin{equation}
u(x,t)=t+t_0 + \xi (\phi(x) - (t+t_0))
\end{equation}
where again, $t_0>0$ and $\xi\in C^\infty(\RN)$ are to be chosen.

Motivated by Property \textit{\ref{subSuper_locUnifConv}}, we set 
\begin{equation}\label{xiSeqS}
\xi(s)=s \quad \textrm{for all }s\leq 0,
\end{equation}
and to obtain Property \textit{\ref{sub_Bounded}}, we will require that 
\begin{equation}\label{xi_at_infty}
\xi(s) \nearrow M\quad\textrm{as } s\to\infty
\end{equation}
for some $M>0$ to be determined.

Computing as in \eqref{supercomp}, we have
\begin{align*}
&{} u_t - \Delta u + |Du|^m - f(x) + \lambda^* \\
&{}\quad \leq -(1-\xi')(f(x) - \lambda^* - 1) + \frac{m-2}{m}\left(\frac{-\xi''}{[\xi' - (\xi')^m]^{\frac{2}{m}}}\right)^\frac{m}{m-2}.
\end{align*}
Combined with \eqref{xiSeqS}, this implies that $u$ is trivially a subsolution for $\phi(x) - (t+t_0) \leq 0$, and using Proposition \ref{prop_chi} we obtain
\begin{equation*}
u_t - \Delta u + |Du|^m - f(x) + \lambda^* \leq -(1-\xi')(f(x) - \lambda^* - 1) + (1-\xi')^\beta,
\end{equation*}
with $0<\beta<1$, $\beta$ arbitrarily close to 1. 

From here on we argue as before to conclude that
\begin{equation*}
U(x,t) := u(x,t) - (1-\beta)\int_0^t (\tau^{\hat{\alpha}} + 1)^{-\hat{\beta}} \,d\tau
\end{equation*}
is a subsolution in all of $\RN\times(0,\infty)$.
\end{proof}

We conclude this section by proving the existence of the functions $\chi$ and $\xi$ needed in the proof of Lemma~\ref{lemma_subSuper}.
\begin{proposition}\label{prop_chi}
Let $\beta \in (0,1)$.
\begin{enumerate}[label=\textit{(\alph*)}]
\item There exist $b>0$ and $\chi\in C^\infty((-\infty,b))$ satisfying \eqref{chi_flat} and \eqref{chiBlows} such that
\begin{align*}
\frac{m-2}{m}\left(\frac{\chi''(s)}{[(1 + \chi'(s))^m - (1 + \chi'(s))]^\frac{2}{m}}\right)^\frac{m}{m-2} \leq (\chi'(s))^\beta\\  
\textrm{ for all } s\in (0, b).
\end{align*}
\item There exists $\xi\in C^\infty(\R)$ satisfying \eqref{xiSeqS} and \eqref{xi_at_infty} such that
\begin{equation*}
\frac{m-2}{m}\left(\frac{-\xi''(s)}{[\xi'(s) - (\xi'(s))^m]^{\frac{2}{m}}}\right)^\frac{m}{m-2} \leq (1-\xi'(s))^\beta \quad\textrm{for all } s\in (0,+\infty).
\end{equation*}
\end{enumerate}
\end{proposition}

\begin{proof}

We begin by proving part \emph{(a)}, then indicate the necessary changes to obtain part \emph{(b).} Motivated by the considerations in the proof of Lemma \ref{lemma_subSuper}, we define $\chi(s) = 0$ for $s\leq 0$ (this is \eqref{xiSeqS}) and for $s>0$ take $\chi$ to be the solution of the ODE
\begin{align}\label{chi_ODE}
&\chi'' = C(\chi')^{\beta_1}(1 + \chi')^{\beta_2}\quad\textrm{in } (0, b),\nonumber\\
&\chi(0)=\chi'(0)=0,
\end{align}
where $\beta_1, \beta_2>0$ are to be chosen, and $b>0$ will determined by $\beta_1$ and $\beta_2$. Note that $\eqref{chi_ODE}$ can be seen as a first-order ODE. Taking $\beta_1<1$ avoids the trivial solution $\chi\equiv 0$, since in this case \eqref{chi_ODE} fails to meet the Osgood condition for uniqueness (see e.g., \cite{agarwal1993uniqueness}). Furthermore, if $\beta_2$ is chosen so that $\beta_1 + \beta_2 > 1$, we achieve the blow-up condition \eqref{chiBlows}. It can also be shown that 
\begin{equation}\label{chi_props}
\chi'(s), \chi''(s)>0\quad\textrm{for all }s>0. 
\end{equation}
In particular, this shows the use of Young's inequality in \eqref{supercomp} is justified.

Using \eqref{chi_ODE}, it remains to prove that for some $C>0$ and any $0<\beta<1$, 
\begin{equation}\label{chiineq2}
\left(\frac{C(\chi'(s))^{\beta_1}(1 + \chi'(s))^{\beta_2}}{\left((1 + \chi'(s))^m - (1+\chi'(s))\right)^\frac{2}{m}}\right)^\frac{m}{m-2} \leq (\chi'(s))^\beta.
\end{equation}
We will proceed by considering different ranges of $s>0$.

Assume first that $s$ is small, say $s<\delta$ for some small $\delta>0$. By the convexity of $s\mapsto s^m$, we have $(1+\chi')^m  - (1+ \chi') \geq (m-1)\chi'$ for $\chi'>0$, and since $(1+\chi')^{\beta_2\frac{m}{m-2}}\rightarrow 1$ for $\chi'\to 0$, we control this factor with the constant $C$ on the left of $\eqref{chiineq2}$. Thus taking a suitably small $C>0$ we have
\begin{align*}
\left(\frac{C(\chi')^{\beta_1}(1+\chi')^{\beta_2}}{\left((1 + \chi')^m - (1+\chi')\right)^\frac{2}{m}}\right)^\frac{m}{m-2} &{}\leq \left(C\frac{(\chi')^{\beta_1}}{(\chi')^\frac{2}{m}}\right)^\frac{m}{m-2}\\
&{}\leq (\chi')^{\frac{\beta_1  m - 2}{m-2}}.
\end{align*}
We thus obtain \eqref{chiineq2} for any $\beta\in (0,1)$ by choosing $\beta_1\in (0,1)$ appropriately. In the proof of Lemma \ref{lemma_subSuper}, we require to have $\beta$ arbitrarily close to $1$. This is achieved by taking $\beta_1$ close to $1$.
Incidentally, the computation also shows that to have control by a linear term in \eqref{chiineq2} (i.e, $\beta=1$) we would require $\beta_1 = 1$, which is impossible if we are to have a nontrivial solution of \eqref{chi_ODE}.

Assume now that $\delta \leq s \leq  b-\delta$. This implies that $\chi'(s)>\bar{\delta}$ for some $\bar{\delta}>0$. In this case, the expressions on either side of \eqref{chiineq2} are continuous, hence remain bounded. Thus \eqref{chiineq2} is obtained by choosing an appropriately small $C$.

Finally, we address the case $s>b-\delta$. Since $\delta>0$ is small and $\chi'(s)\to+\infty$ as $s\to b^-$, this amounts to checking \eqref{chiineq2} in the the limit $\chi'\to\infty$. Using only that $m>2$, a straightforward computation shows that setting $1<\beta_2<2$ in \eqref{chi_ODE}, the left-hand side of \eqref{chiineq2} vanishes as $\chi'\to+\infty$, while the right-hand side goes to infinity. We have thus shown \eqref{chiineq2} holds for all $s\in (0,b)$. Using \eqref{chi_flat}, we conclude for all $s\in (-\infty, b)$.


For part \emph{(b)}, define $\xi(s)$ for $s>0$ as the solution of
\begin{align}\label{xiODE}
\xi'' = -C(1-\xi')^{\eta_1}(\xi')^{\eta_2}\quad\textrm{in }(0,+\infty),\qquad \xi(0)=0,\ \xi'(0)=1,
\end{align}
with $0<\eta_1<1$, $\eta_2>0$ to be chosen. 

It can be shown that a nontrivial solution exists and satisfies
\begin{equation}\label{xi_props}
\xi(s)>0,\ -C\leq\xi''(s)<0\quad\textrm{and}\quad0<\xi'(s)<1, \qquad\textrm{for all }s>0,
\end{equation} 
for some $C>0$, while taking $\eta_1 + \eta _2 > 1$ gives \eqref{xi_at_infty}.

An  analysis similar to that of part \emph{(a)} gives that \eqref{xiODE} and \eqref{xi_props} imply
\begin{equation*}
\left(\frac{-\xi''}{[\xi' - (\xi')^m]^{\frac{2}{m}}}\right)^\frac{m}{m-2} \leq (1-\xi')^{\frac{\eta_1m-2}{m-2}}.
\end{equation*}
Thus, choosing $\eta_1$ appropriately we conclude.
\end{proof}

%

\section{Large-time behavior}\label{sec_LT}

In this final section, we use the existence of sub- and supersolutions given by Lemma~\ref{lemma_subSuper} to perform the two steps of the convergence proof.

\begin{lemma}\label{lemma_unifBounded}
Assume $f\in \locLip(\RN)$ satisfies \emph{\ref{growth_f}} and $u_0 \in C(\RN)$ is bounded from below. Then $u(x,t) - \lambda^*t$ is bounded over compact sets, uniformly with respect to $t>0$.
\end{lemma}

\begin{proof}
The proof follows by comparing $u(x,t)-\lambda^*t$, which solves \eqref{mod_ParabErg}, to the sub- and supersolutions constructed in Lemma \ref{lemma_subSuper}, to which we refer the reader for notation and properties. 

Let $K$ be a compact subset of $\RN$. As noted earlier, if $t_0>0$ is large enough, we have $K\subset Q_0 = \overline{\{\phi(x) - t_0 < b \}}$, while $Q_0$ is also a compact subset of $\RN$, by the coercivity of $\phi$ (see Lemma \ref{lemma_superlinPhi}). Hence, for a large enough $C\in \R$, we have
\begin{equation*}
V(x,0) + C \geq u_0\quad\textrm{ in }Q_0.
\end{equation*}
Moreover, by construction, we have $K\times (0,+\infty) \subset Q$.

Since, by construction, $V(x,t)\to +\infty$ on the lateral boundary of $Q$, by comparison (Theorem \ref{thm_comparison}), we have that
\begin{equation*}
V(x,t) + C \geq u(x,t) - \lambda^*t\quad \textrm{for all }(x,t)\in Q.
\end{equation*}
Hence, this inequality is true on $K\times (0,+\infty)$.

Furthermore, by Lemma \ref{lemma_subSuper} \emph{\ref{subSuper_locUnifConv}},
\begin{equation*}
V(x,t) + C \to \phi(x) + \sigma + C, \quad\textrm{uniformly over }K\textrm{ as }t\to+\infty. 
\end{equation*}
Therefore, we have
\begin{equation*}
\limsup_{t\to+\infty}\ \!\!\!^*\,  \left(u(x,t) - \lambda^*t \right) 
\leq \phi(x) + \sigma + C \quad \textrm{for all }(x,t)\in K\times (0,+\infty).
\end{equation*}

Similarly, we use the subsolution $U$ from Lemma \ref{lemma_subSuper} to obtain the lower bound. Recalling that $u_0$ may be assumed nonnegative, by Lemma \ref{lemma_subSuper} \emph{\ref{sub_Bounded}}, we have
\begin{equation*}
U(x,0) - M \leq 0 \leq u_0(x) \quad\textrm{for all } x\in \RN,
\end{equation*}
hence, by comparison,
\begin{equation*}
U(x,t) - M \leq u(x,t) - \lambda^*t \quad\textrm{for all }(x,t)\in \RN\times(0,+\infty).
\end{equation*}
Thus for large $t>0$, by Lemma \ref{lemma_subSuper} \emph{\ref{subSuper_locUnifConv}},
\begin{equation*}
\phi(x) - \sigma - M \leq u(x,t) - \lambda^*t\quad\textrm{for all }x\in K,
\end{equation*}
and finally,
\begin{equation*}
\phi(x) - \sigma - M - C \leq \liminf_{t\to+\infty}\ \!\!\!_*\,  \left(u(x,t) - \lambda^*t \right)\quad\textrm{for all }x\in K, t>0,
\end{equation*}
by taking $C>0$ as before.
\end{proof}

An immediate consequence of Lemma \ref{lemma_unifBounded} is the following weaker convergence result.
\begin{corollary}
Under the assumptions of Lemma \ref{lemma_unifBounded},
\begin{equation*}
\frac{u(x,t)}{t}\to \lambda^*\quad\textrm{locally uniformly in }\RN\textrm{ as }t\to+\infty.
\end{equation*}
\end{corollary}

\subsection{Main result}
The rest of this section is devoted to the proof of Theorem \ref{thm_fullConv}, stated in the introduction.

\begin{proof}[Proof of Theorem \ref{thm_fullConv}]

\emph{Step 1.} For simplicity we write $v(x,t):=u(x,t) - \lambda^*t$. By Lemma \ref{lemma_unifBounded}, $v$ is locally bounded for all $t>0$, hence the half-relaxed limit
\begin{equation*}
\bar{v}(x) = \limsup_{t\to\infty}\ \!\!\!^*\, v(x,t)
\end{equation*}
is well-defined for all $x\in\RN$. By the stability of viscosity solutions, $\bar{v}$ is a subsolution of \eqref{erg_prob_intro} in all of $\RN$.  Furthermore, adding an appropriate constant to either $\phi$ or $\bar{v}$ so that they coincide at some point, we have by the Strong Maximum Principle that 
\begin{equation*}
\bar{v}(x) = \phi(x) + \hat{c} \quad\textrm{for all }x\in\RN,
\end{equation*}
for some $\hat{c}\in\RN$ (see, e.g., the proof of Theorem 3.1 in \cite{barles2016unbounded} for details).

\emph{Step 2.} Fix $\hat{x}\in\RN$. By the definition of half-relaxed limits, there exists a sequence $(x_n, t_n)\in\RN\times(0,+\infty)$ such that $x_n\to \hat{x}, \ t_n\to\infty$, and $v(x_n,t_n)\to\bar{v}(\hat{x})$. Consider $v_n(\cdot):=v(\cdot, t_{n} - 1)$. Again by Lemma \ref{lemma_unifBounded}, the sequence $(v_n)$ is uniformly bounded over compact sets. Furthermore, by the local gradient bound of Theorem \ref{gradBound} (in the appendix), it is also uniformly equicontinuous over compact sets. Thus, there exists $w_0\in C(\RN)$ such that, given a compact $K\subset \RN$, there exists a subsequence $(v_{n'})$ such that $v_{n'}\to w_0$ uniformly over $K$ as $n'\to\infty$.

Consider now
\begin{equation*}
w_{n'}(x,t) = v(x, t + t_{n'} - 1)\quad\textrm{for }(x,t)\in\RN\times(0,+\infty).
\end{equation*}
The sequence $(w_{n'})$ is again uniformly equicontinuous (in both space and time variables) due to Corollary \ref{gradBoundCor}. Thus, given also $T>0$, we have that 
\begin{equation}\label{defW}
w_{n'}\to w \textrm{ uniformly over } K\times[0,T] \quad\textrm{ for some }w \in C(\RN\times (0,+\infty))
\end{equation}
(again passing to a subsequence if necessary---this is omitted for ease of notation). By the definition of the half-relaxed limit $\bar{v}$, $w(x,t)\leq\bar{v}(x)=\phi(x) + \hat{c}$ for all $(x,t)\in\RN\times(0,+\infty)$, and by construction,
\begin{equation}
w(\hat{x},1) = \lim_{n'} v(\hat{x},t_{n'}) = \phi(\hat{x}) + \hat{c}.
\end{equation}
Hence, by the parabolic Strong Maximum Principle (see Lemma \ref{strongMaxP} and Remark \ref{strongMaxP_rmk} in the appendix), $w - (\phi + \hat{c})$ is constant in $\RN\times[0,1]$. In particular, $w_0(x)=\phi(x) + \hat{c}$ for all $x\in\RN$. 

To summarize, we have obtained that given any compact $K\subset\RN$, there exists a sequence $t_{n'} - 1\to\infty$ such that
\begin{equation}\label{unifConvSubseq}
v(\cdot, t_{n'} - 1)\to\phi+\hat{c}\quad\textrm{ uniformly over }K\textrm{ as }t_{n'} -1 \to\infty.
\end{equation}
In the next step, we will use \eqref{unifConvSubseq} for a suitable $K$ that is chosen larger than the set on which the uniform convergence will hold.


\emph{Step 3.} Let $\epsilon>0$ and $\widehat{K}\subset\RN$ be compact (this is the set on which we will prove the uniform convergence stated in the Theorem). For this final part of the proof we employ many of the elements of Lemma \ref{lemma_subSuper}, to which we refer the reader.

For $R>0$ and $(x,t)\in\RN\times(0,+\infty)$, we define
\begin{equation*}
V_R(x,t) = \phi(x) + \hat{c} + \chi(\phi(x) + \hat{c} - (t+R)) + \int_R^{t+R} (\tau^{\hat{\alpha}} + 1)^{-\hat{\beta}}\,d\tau + \frac{1}{R},
\end{equation*}
with $\chi, \hat{\alpha}$ and  $\hat{\beta}$ are as in the definition of the supersolution $V$. Thus, arguing as in the proof of Lemma \ref{lemma_subSuper}, $V_R$ is a supersolution of \eqref{mod_ParabErg} in $Q=\{\phi(x) + \hat{c} - (t + R) < b \}$.



We take $R>0$ large enough so that
	\begin{equation}\label{notSoBigRsuper}
		\phi(x) + \hat{c} - R < b \quad \textrm{for all } x\in \hat{K},
	\end{equation}
where $b$ is given by \eqref{chiBlows}, i.e., such that $\chi(s) \to +\infty$ for $s \to b^-$. Thus, by \eqref{notSoBigRsuper} we have that
$\hat{K}\subset Q_0 = \overline{Q}\cap \{t=0\}$, and thus $\hat{K} \times (0,+\infty) \subset Q$. Furthermore, arguing as in the proof of Lemma \ref{lemma_subSuper},
	\begin{equation}\label{VR_blows}
		V_R(y,0)\to +\infty\quad\textrm{as }y \to x,\textrm{ for all }x\in Q_0,
	\end{equation}
and 
	\begin{equation}\label{VR_blows_2}
		V_R(y,s)\to +\infty\quad\textrm{as }(y,s)\to (x,t),\textrm{ for all } (x,t)\in \partial_p Q, \ t>0.
	\end{equation}
Recall that, as a consequence of Lemma \ref{lemma_superlinPhi}, $Q_0$ is compact (see also the proof of Lemma \ref{lemma_unifBounded}). We can thus use \eqref{unifConvSubseq} from Step 2 for $K=Q_0$ to obtain that, for large enough $n'$,
	\begin{equation*}
		v(x,t_{n'} - 1) < \phi(x) + \hat{c} + \frac{1}{R} \quad\textrm{for all }x\in Q_0
	\end{equation*}
This gives that $v(x,t_{n'} - 1) \leq V_R(x,0)$ in $Q_0$, by construction. 
Together with \eqref{VR_blows}, \eqref{VR_blows_2}, this implies
	\begin{equation}\label{VR_at_parBoundary}
		v(x,t_{n'} - 1 + t) \leq V_R(x,t)\quad\textrm{for all }(x,t) \in \partial_p Q.
	\end{equation}

Since $v$ and $V_R$ are a solution and a supersolution of \eqref{mod_ParabErg}, respectively, and satisfy \eqref{VR_at_parBoundary}, by comparison (Theorem \ref{thm_comparison}) we have 
	\begin{equation}\label{VR_compared}
		v(x,t + t_{n'}-1) \leq V_R(x,t)\quad\textrm{in }Q.
	\end{equation}
In particular, \eqref{VR_compared} holds in $\hat{K}\times (0,+\infty)$.

We note that $\int_R^{+\infty} (\tau^{\hat{\alpha}} + 1)^{-\hat{\beta}}\,d\tau + \frac{1}{R} = o(1)$ as $R\to+\infty$. Thus, arguing as in the proof of Lemma \ref{lemma_subSuper} \textit{\ref{subSuper_locUnifConv}}, we take $R>0$ larger still so that,
\begin{align}\label{VR_on_K}
\widehat{V}_R(x) :={}&\phi(x) + \hat{c} + \chi(\phi(x) + \hat{c} - R) + \int_R^{+\infty} (\tau^{\hat{\alpha}} + 1)^{-\hat{\beta}}\,d\tau + \frac{1}{R}\nonumber\\
 <{}&\phi(x) + \hat{c} + \epsilon\quad \textrm{for all }x\in \widehat{K}.
\end{align}
By \eqref{chi_props}, $t\mapsto\chi(\phi(x) + \hat{c} - (t+R))$ is nonincreasing, hence 
\begin{equation}\label{V_R_monoton}
V_R(x,t)\leq\widehat{V}_R(x)\quad\textrm{ for all }t>0,\ \hat{x}\in\widehat{K}. 
\end{equation}
Therefore, combining \eqref{VR_compared}, \eqref{VR_on_K} and \eqref{V_R_monoton}, we obtain
\begin{equation}\label{upper_Conv_Bound}
v(x,t + t_{n'} - 1) \leq \phi(x) + \hat{c} + \epsilon \quad \textrm{for all }x\in \widehat{K}, \ t>0.
\end{equation}

To obtain the lower bound corresponding to \eqref{upper_Conv_Bound}, we define the analogue of $V_R$ based on the subsolution $U$ from Lemma \ref{lemma_subSuper}. For $(x,t)\in\RN\times(0,+\infty)$, define
\begin{equation*}
U_R(x,t) = t + R + \xi(\phi(x) + \hat{c} - (t+R)) - \int_{R}^{t+R} (\tau^{\hat{\alpha}} + 1)^{-\hat{\beta}}\,d\tau - \frac{1}{R},
\end{equation*}
with $\xi, \hat{\alpha}$ and $\hat{\beta}$ as before. Consider now $\underline{v}(x) = \liminf_{t\to\infty}\ \!\!\!_* \,v(x,t).$ Arguing as in Step 1, it follows that
\begin{equation*}
\underline{v}(x) = \phi(x) + m^-,
\end{equation*}
for some $m^-\in\R$. By definition of the half-relaxed limit, we have
\begin{equation}\label{lowerBound_vn}
\phi(x) + m^- \leq v(x,t_{n'}-1) \quad\textrm{for all }x\in \RN.
\end{equation}
for sufficiently large $n'$. Recall from the construction of Lemma \ref{lemma_subSuper} that $\xi(s)\leq M$ for all $s\in\R$, for some $M\in \R$. We set $R>0$ large enough so that $\hat{K}\subset B_R$, $U_R$ is a subsolution of \eqref{mod_ParabErg} and, using Lemma \ref{lemma_superlinPhi},
\begin{equation}\label{bigRsub} 
\phi(x) - R \geq M-m^- \quad\textrm{for all }x\in \RN\backslash B_R,
\end{equation}
in addition to the requirements made in the construction of $V_R$. We then have, by \eqref{lowerBound_vn} and \eqref{bigRsub},
\begin{align*}
U_R(x,0) \leq{}& R + \xi(\phi(x) + \hat{c} - R) \leq R + M \nonumber\\
\leq{}&\phi(x) + m^- \leq v(x, t_{n'}-1)\quad\textrm{for all }x\in\RN\backslash B_R.
\end{align*}
By \eqref{unifConvSubseq} we have, for sufficiently large $n'$,
\begin{equation*}
v(x,t_{n'}-1) > \phi(x) + \hat{c} - \frac{1}{R} \quad\textrm{for all }x\in \overline{B}_R
\end{equation*}
Recalling that $\xi(s)\leq s$ for all $s\in\R$, we have
\begin{equation*}
U_R(x,0) \leq R + \xi(\phi(x) + \hat{c} - R) -\frac{1}{R} \leq \phi(x) + \hat{c} - \frac{1}{R} \quad\textrm{for all }x\in\RN.
\end{equation*}
Thus we obtain
\begin{equation}\label{UR_at_0}
U_R(x,0)\leq v(x,t_{n}-1)\quad\textrm{for all }x\in\RN.
\end{equation}
By comparison (Theorem \ref{thm_comparison}), this implies that
\begin{equation}
U_R(x,t) \leq v(x,t +t_{n}-1)\quad\textrm{for all }x\in\RN, \ t>0.
\end{equation}
From this point on, we argue as we did before for $V_R$. We remark that the analogue of \eqref{V_R_monoton} (for a similarly defined $\widehat{U}_R(x)$) is now given by the fact that the function $t \mapsto t + R + \xi(\phi(x) + \hat{c} - (t + R))$ is nondecreasing, since $0\leq\xi'(s)\leq 1$ for all $s\in\R$, by \eqref{xi_props}. Thus, for large enough $n'$, depending on $R>0$, we have
\begin{equation*}
v(x,t+t_{n'}-1)\geq \phi(x) + \hat{c} - \epsilon, \quad \textrm{for all }x\in \widehat{K}, \ t>0,
\end{equation*}
and with this we conclude.
\end{proof}

\appendix
\section{Appendix}
In this appendix we present some estimates and results used in the previous sections.
\begin{theorem}[Hölder estimate]\label{thm_holder}
For $R>0$, let $u\in USC (\overline{B}_R)$ be a subsolution of 
\begin{equation*}
-\mathrm{tr}(A(x)D^2u) + |Du|^m = f_1(x) \quad\textrm{in } B_R,
\end{equation*}
where for each $x\in B_R$, $A(x)$ is a nonnegative symmetric matrix such that the map $x\mapsto A(x)$ is bounded and continuous in $B_R$, and $f_1\in C(B_R)$. Then $u\in C^{0,\gamma}(\overline{B}_R)$ and
\begin{equation*}
|u(x) - u(y)| \leq K_1|x-y|^\gamma,
\end{equation*}
where $\gamma = \frac{m-2}{m-1}$ and $K_1$ depends only on $m, \|A\|_{L^\infty(B_R)}$ and $\|d_{\partial B_R}^{m(1-\gamma)} f_1^+\|_{L^\infty(B_R)}$.
\end{theorem}

The Theorem follows immediately from Lemmas 2.1 and 2.2 in \cite{capuzzo2010holder} (see also \cite{barles2010short}). For equations set on a bounded domain, these lemmas lead to global Hölder estimates assuming the boundary is sufficiently regular. The result as stated is sufficient for our purposes.

More importantly, we remark that a crucial feature of the above estimate is that it depends on $\|A\|_{L^\infty(B_R)}$, but not on any \emph{lower bound} for the matrix $A$. It is in fact valid in the completely degenerate, or first-order, case.

\begin{theorem}[Local Gradient Bounds]\label{gradBound}
Let $R, \tau>0$.
	\begin{enumerate}[label=(\textit{\alph*})]
	\item\label{gradBerg} There exists $K_2>0$ depending only on $m$ and $N$, such that for any $R \geq R' + 1 > 0$ the solution of \eqref{erg_prob_intro} satisfies
		\begin{equation*}
			\sup_{B_{R'}} |D\phi| \leq K_2(1 + \sup_{B_R} |f|^{\frac{1}{m}} + \sup_{B_R} |Df|^{\frac{1}{2m-1}}).
		\end{equation*}
	
	\item\label{gradBparab} If $u$ is a solution of
		\begin{align}
			u_t - \Delta u + |Du|^m = f(x) &{}\quad\textrm{in } \Omega \times (0,T], \label{HJO}\\
			u(x, 0) = u_0(x) &{}\quad\textrm{in } \overline{\Omega},\label{HJO-i}
		\end{align}
	where $\Omega$ is a domain of $\R^N$ such that $B_{R+1}\subset\Omega$ and $f\in \locLip(\RN)$, then $u$ is Lipschitz continuous in $x$ in $B_R\times[\tau, +\infty)$ and $|Du(x,t)|\leq L$ for a.e. $x\in B_R$, for all $t\geq\tau$, where $L$ depends on $R$ and $\tau$. Moreover this result holds with $\tau=0$ if $u_0$ is locally Lipschitz continuous in $\Omega$.
	\end{enumerate}
\end{theorem}

Both results in Theorem \ref{gradBound} are classical. The estimate in \ref{gradBerg} appears as stated in \cite{ichihara2012large}, but can also be inferred from the results of \cite{lasry1989nonlinear} (see also \cite{lions1980resolution}, \cite{lions1985quelques}).

The conclusion of \ref{gradBparab} can also be adapted from the results of \cite{lions1982generalized} (that recovers some of the results from \cite{lions1980resolution}). For a proof closer to our setting---namely, within the context of viscosity solutions, via the \emph{weak 
Bernstein method}---we refer the reader to Theorem 4.1 in \cite{barles2017local}. The viscous Hamilton-Jacobi Equation (\ref{HJO}) is easily shown to satisfy the structure conditions required therein. 

Furthermore, we remark that in this last reference the estimate obtained holds for an equation satisfied in $(0,T)$ (for arbitrary $T>0$), but has no dependence on the data at $t=0$ and can therefore be extended to $[\tau, +\infty)$ for $\tau>0$.

	\begin{corollary}\label{gradBoundCor}
		Let $R, \tau>0$. The solution $u$ of \eqref{HJO}-\eqref{HJO-i} is Hölder-continuous of order $1/2$ in $B_R\times[\tau, +\infty)$ and $\|u\|_{C^{0,1/2}(B_R\times[\tau, +\infty))}\leq M$ for some $M>0$ depending on $R, \tau$, the data $u_0, f$ and universal constants. Moreover, this result holds with $\tau=0$ if $u_0$ is locally Lipschitz continuous in $\Omega$.
	\end{corollary}
	
	\begin{proof} The proof of Corollary~\ref{gradBoundCor} can be done in two ways: either by using classical interior parabolic estimates (see \cite{wang1992regularity}, Theorem 4.19, and also Theorem 4.36 in \cite{imbert2013introduction}), in which case the Hölder-regularity of the solution $u$ is of some order $\alpha\in (0,1)$ depending on universal constants, or by the argument of \cite{BBL}, which implies that a solution which is Lipschitz in $x$ is $1/2$-Hölder-continuous in $t$.
	
	In both proofs, the result relies on part \ref{gradBparab} of Theorem \ref{gradBound}, which implies that the solution $u$ of \eqref{vhj_whole} satisfies $|Du|\leq L$ in $B_{R}\times (\tau, +\infty)$, and therefore that $u_t - \Delta u $ is bounded in $B_{R}\times (\tau, +\infty)$. This allows the use of either of the two arguments just mentioned.
	\end{proof}

\begin{lemma}[Strong Maximum Principle]\label{strongMaxP}
	Let $R,C>0$. Any upper semicontinuous subsolution of
		\begin{equation}\label{strongMaxP_eq}
			u_t - \Delta u - C|Du| = 0 \quad\text{in } B_R \times (0, +\infty)
		\end{equation}
	that attains its maximum at some $(x_0, t_0)\in B_R \times (0, +\infty)$ is constant in $B_R \times [0, t_0]$.
\end{lemma}

We refer the reader to, e.g., \cite{da2004remarks}, Corollary 2.4 and \cite{bardi1999strong}, Corollary 1. (The latter result concerns time-independent equations, but the method of proof equally applies to this context.)

\begin{remark}\label{strongMaxP_rmk}
	The difference $w-\phi$, where $\phi$ is a solution of the ergodic equation \eqref{erg_prob_intro} and $w$ is given by \eqref{defW}, can be shown to satisfy an equation like \eqref{strongMaxP_eq} by using the convexity of $\xi\to|\xi|^m$ in any ball $B_R$. In this case, the constant $C$ in \eqref{strongMaxP_eq} depends on the gradient bound for $\phi$ from \emph{Theorem \ref{gradBound}, \ref{gradBerg}}. Of course, the complete result is obtained by letting $R$ tend to $+\infty$.
\end{remark}

\noindent\textbf{Acknowledgments:} G.B.~was partially supported by the ANR MFG
(ANR-16-CE40-0015-01). A.Q.~was partially supported by Fondecyt Grant Nº 1190282 and Programa Basal, CMM, U.~de Chile. A.R.~was partially supported by Fondecyt, Postdoctorado 2019, Proyecto Nº 3190858.

\bibliographystyle{plain}
\bibliography{aerp.bib}

\begin{thebibliography}{10}

\bibitem{agarwal1993uniqueness}
Ravi~P Agarwal and Vangipuram Lakshmikantham.
\newblock {\em Uniqueness and nonuniqueness criteria for ordinary differential
  equations}, volume~6.
\newblock World Scientific Publishing Company, 1993.

\bibitem{arapostathis2019uniqueness}
Ari Arapostathis, Anup Biswas, and Luis Caffarelli.
\newblock On uniqueness of solutions to viscous hjb equations with a
  subquadratic nonlinearity in the gradient.
\newblock {\em Communications in Partial Differential Equations},
  44(12):1466–1480, Jul 2019.

\bibitem{bardi2008optimal}
Martino Bardi and Italo Capuzzo-Dolcetta.
\newblock {\em Optimal control and viscosity solutions of
  Hamilton-Jacobi-Bellman equations}.
\newblock Springer Science \& Business Media, 2008.

\bibitem{bardi1999strong}
Martino Bardi and Francesca Da~Lio.
\newblock On the strong maximum principle for fully nonlinear degenerate
  elliptic equations.
\newblock {\em Archiv der Mathematik}, 73(4):276--285, 1999.

\bibitem{barles2017local}
G~Barles.
\newblock Local gradient estimates for second-order nonlinear elliptic and
  parabolic equations by the weak bernstein's method.
\newblock {\em arXiv preprint arXiv:1705.08673}, 2017.

\bibitem{BS}
G.~Barles and Panagiotis Souganidis.
\newblock Space-time periodic solutions and long-time behavior of solutions to
  quasilinear parabolic equations.
\newblock {\em Siam Journal on Mathematical Analysis - SIAM J MATH ANAL}, 32,
  03 2001.

\bibitem{barles2010short}
Guy Barles.
\newblock A short proof of the ${C}^{0,\alpha}$-regularity of viscosity
  subsolutions for superquadratic viscous {H}amilton--{J}acobi equations and
  applications.
\newblock {\em Nonlinear Analysis: Theory, Methods \& Applications},
  73(1):31--47, 2010.

\bibitem{BBL}
Guy {Barles}, Samuel {Biton}, and Olivier {Ley}.
\newblock {Uniqueness for parabolic equations without growth condition and
  applications to the mean curvature flow in \(\mathbb R^2\).}
\newblock {\em {J. Differ. Equations}}, 187(2):456--472, 2003.

\bibitem{barles2004generalized}
Guy Barles and Francesca Da~Lio.
\newblock On the generalized {D}irichlet problem for viscous
  {H}amilton--{J}acobi equations.
\newblock {\em Journal de Math{\'e}matiques Pures et Appliqu{\'e}es},
  83(1):53--75, 2004.

\bibitem{barles2002convergence}
Guy Barles and Espen~Robstad Jakobsen.
\newblock On the convergence rate of approximation schemes for
  {H}amilton-{J}acobi-{B}ellman equations.
\newblock {\em ESAIM: Mathematical Modelling and Numerical Analysis},
  36(1):33--54, 2002.

\bibitem{barles2017large}
Guy Barles, Olivier Ley, Thi-Tuyen Nguyen, and Thanh Phan.
\newblock Large time behavior of unbounded solutions of first-order
  {H}amilton--{J}acobi equations in $\mathbb{R}^n$.
\newblock {\em arXiv preprint arXiv:1709.08387}, 2017.

\bibitem{barles2016unbounded}
Guy Barles and Joao Meireles.
\newblock On unbounded solutions of ergodic problems in $\mathbb{R}^m$ for
  viscous {H}amilton--{J}acobi equations.
\newblock {\em Communications in Partial Differential Equations},
  41(12):1985--2003, 2016.

\bibitem{barles2010large}
Guy Barles, Alessio Porretta, and Thierry~Tabet Tchamba.
\newblock On the large time behavior of solutions of the {D}irichlet problem
  for subquadratic viscous {H}amilton--{J}acobi equations.
\newblock {\em Journal de math{\'e}matiques pures et appliqu{\'e}es},
  94(5):497--519, 2010.

\bibitem{barles2006ergodic}
Guy Barles and Jean-Michel Roquejoffre.
\newblock Ergodic type problems and large time behaviour of unbounded solutions
  of {H}amilton--{J}acobi equations.
\newblock {\em Communications in Partial Differential Equations},
  31(8):1209--1225, 2006.

\bibitem{barles2001space}
Guy Barles and Panagiotis~E Souganidis.
\newblock Space-time periodic solutions and long-time behavior of solutions to
  quasi-linear parabolic equations.
\newblock {\em SIAM Journal on Mathematical Analysis}, 32(6):1311--1323, 2001.

\bibitem{benachour2004asymptotic}
Said Benachour, Grzegorz Karch, and Philippe Lauren{\c{c}}ot.
\newblock Asymptotic profiles of solutions to viscous hamilton--jacobi
  equations.
\newblock {\em Journal de math{\'e}matiques pures et appliqu{\'e}es},
  83(10):1275--1308, 2004.

\bibitem{biler2004asymptotic}
Piotr Biler, Grzegorz Karch, and Mohammed Guedda.
\newblock Asymptotic properties of solutions of the viscous hamilton-jacobi
  equation.
\newblock {\em Journal of Evolution Equations}, 4(1):75--97, 2004.

\bibitem{capuzzo2010holder}
I~Capuzzo~Dolcetta, Fabiana Leoni, and Alessio Porretta.
\newblock H{\"o}lder estimates for degenerate elliptic equations with coercive
  {H}amiltonians.
\newblock {\em Transactions of the American Mathematical Society},
  362(9):4511--4536, 2010.

\bibitem{crandall1992user}
Michael~G Crandall, Hitoshi Ishii, and Pierre-Louis Lions.
\newblock User's guide to viscosity solutions of second order partial
  differential equations.
\newblock {\em Bulletin of the American Mathematical Society}, 27(1):1--67,
  1992.

\bibitem{da2004remarks}
Francesca Da~Lio.
\newblock Remarks on the strong maximum principle for viscosity solutions to
  fully nonlinear parabolic equations.
\newblock {\em Communications on Pure and Applied Analysis}, 3:395--416, 2004.

\bibitem{dall2015local}
Andrea Dall’Aglio and Alessio Porretta.
\newblock Local and global regularity of weak solutions of elliptic equations
  with superquadratic hamiltonian.
\newblock {\em Transactions of the American Mathematical Society},
  367(5):3017--3039, 2015.

\bibitem{evans1998partial}
L.C. Evans.
\newblock {\em Partial Differential Equations}.
\newblock Graduate studies in mathematics. American Mathematical Society, 1998.

\bibitem{fleming2006controlled}
Wendell~H Fleming and Halil~Mete Soner.
\newblock {\em Controlled Markov processes and viscosity solutions}, volume~25.
\newblock Springer Science \& Business Media, 2006.

\bibitem{gallay2007asymptotic}
Thierry Gallay and Philippe Lauren{\c{c}}ot.
\newblock Asymptotic behavior for a viscous hamilton-jacobi equation with
  critical exponent.
\newblock {\em Indiana University mathematics journal}, pages 459--479, 2007.

\bibitem{iagar2010asymptotic}
Razvan~Gabriel Iagar, Philippe Lauren{\c{c}}ot, and Juan~Luis V{\'a}zquez.
\newblock Asymptotic behaviour of a nonlinear parabolic equation with gradient
  absorption and critical exponent.
\newblock {\em arXiv preprint arXiv:1002.2094}, 2010.

\bibitem{ichihara2012large}
Naoyuki Ichihara.
\newblock Large time asymptotic problems for optimal stochastic control with
  superlinear cost.
\newblock {\em Stochastic Processes and their Applications}, 122(4):1248--1275,
  2012.

\bibitem{ichihara2009long}
Naoyuki Ichihara and Hitoshi Ishii.
\newblock Long-time behavior of solutions of hamilton--jacobi equations with
  convex and coercive hamiltonians.
\newblock {\em Archive for rational mechanics and analysis}, 194(2):383--419,
  2009.

\bibitem{imbert2013introduction}
Cyril Imbert and Luis Silvestre.
\newblock An introduction to fully nonlinear parabolic equations.
\newblock In {\em An introduction to the K{\"a}hler-Ricci flow}, pages 7--88.
  Springer, 2013.

\bibitem{ishii2008asymptotic}
Hitoshi Ishii.
\newblock Asymptotic solutions for large time of hamilton-jacobi equations in
  euclidean $ n $ space.
\newblock In {\em Annales de l'IHP Analyse non lin{\'e}aire}, volume~25, pages
  231--266, 2008.

\bibitem{ishii2009asymptotic}
Hitoshi Ishii.
\newblock Asymptotic solutions of hamilton-jacobi equations for large time and
  related topics.
\newblock In {\em ICIAM 07—6th International Congress on Industrial and
  Applied Mathematics}, pages 193--217, 2009.

\bibitem{lasry1989nonlinear}
Jean-Marie Lasry and Pierre-Louis Lions.
\newblock Nonlinear elliptic equations with singular boundary conditions and
  stochastic control with state constraints.
\newblock {\em Mathematische Annalen}, 283(4):583--630, 1989.

\bibitem{laurenccot2009non}
Philippe Lauren{\c{c}}ot.
\newblock Non-diffusive large time behavior for a degenerate viscous
  hamilton--jacobi equation.
\newblock {\em Communications in Partial Differential Equations},
  34(3):281--304, 2009.

\bibitem{lions1980resolution}
Pierre-Louis Lions.
\newblock R{\'e}solution de problemes elliptiques quasilin{\'e}aires.
\newblock {\em Archive for Rational Mechanics and Analysis}, 74(4):335--353,
  1980.

\bibitem{lions1982generalized}
Pierre-Louis Lions.
\newblock {\em Generalized solutions of Hamilton-Jacobi equations}, volume~69.
\newblock London Pitman, 1982.

\bibitem{lions1985quelques}
Pierre-Louis Lions.
\newblock Quelques remarques sur les probl{\`e}mes elliptiques
  quasilin{\'e}aires du second ordre.
\newblock {\em Journal d'Analyse Math{\'e}matique}, 45(1):234--254, 1985.

\bibitem{tabet2010large}
Thierry Tabet~Tchamba.
\newblock Large time behavior of solutions of viscous {H}amilton--{J}acobi
  equations with superquadratic {H}amiltonian.
\newblock {\em Asymptotic Analysis}, 66(3-4):161--186, 2010.

\bibitem{wang1992regularity}
Lihe Wang.
\newblock On the regularity theory of fully nonlinear parabolic equations: I.
\newblock {\em Communications on pure and applied mathematics}, 45(1):27--76,
  1992.

\end{thebibliography}

\vspace{5mm}
\noindent\textsc{Guy Barles}\\
\noindent\textit{Email:} guy.barles@idpoisson.fr\\
\noindent\textsc {Institut Denis Poisson (UMR CNRS 7013), Université de Tours, Université d'Orléans, CNRS, Parc de Grandmont 37200 Tours, France}\\[4pt]

\noindent \textsc{Alexander Quaas}\\
\noindent \textit{Email:} alexander.quaas@usm.cl\\
\noindent \textsc{Departamento de Matemática, Universidad Técnica Federico Santa María, Casilla V-110, Avda.~España 1680, Valparaíso, Chile.}\\[4pt]

\noindent \textsc{Andrei Rodríguez}\\
\noindent \textit{Email:} andrei.rodriguez@usach.cl\\
\noindent \textsc{Departamento de Matemática y Ciencia de la Computación, Universidad de Santiago de Chile, Avda.~Libertador General Bernardo O'Higgins 3383, Santiago, Chile.}
\end{document}